\documentclass[a4paper,12pt]{article}
\usepackage{amsmath,amsthm,amsfonts}
\usepackage{fullpage,comment,graphicx,epstopdf}
\usepackage{multicol}

\usepackage{color,graphicx,times}

\input{xy}
\xyoption{all}
\xyoption{matrix}

\theoremstyle{plain}

\newtheorem{teo}{Theorem}
  \newtheorem{lem}[teo]{Lemma}
  \theoremstyle{remark}
  \newtheorem{rem}[teo]{Remark}
 
\theoremstyle{definition}
  
  \theoremstyle{definition}
  \newtheorem{defn}[teo]{Definition}
    \newtheorem{defi}[teo]{Definition}

  \newtheorem{ex}[teo]{Example}
  \theoremstyle{plain}
  \newtheorem{prop}[teo]{Proposition}
  \theoremstyle{plain}
  
  \newtheorem{coro}[teo]{Corollary}

 \def\be{\beta} 
\def\t{\triangleleft}

\def\wt{\widetilde}

\def\Aut{\mathrm{Aut}}

\def\id{\mathrm{Id}}

\def\Z{\mathbb{Z}}

\def\m1{^{ \hbox{\small{-}}1}}

\def\Uncfg{U_{nc}^{fg}}

\title{Virtual link and knot invariants from non-abelian Yang-Baxter 
2-cocycle pairs}

\author{Marco A. Farinati\thanks{Member of CONICET. Partially supported by
PIP 11220110100800CO, and UBACYT 20021030100481BA, mfarinat@dm.uba.ar.} 
\ and Juliana Garc\'ia Galofre\thanks{Partially supported by 
PIP 11220110100800CO and UBACYT 20021030100481BA,
jgarciag@dm.uba.ar }
}
\begin{document}
\maketitle
\begin{abstract}
For a given $(X,S,\beta)$, where $S,\beta\colon X\times X\to X\times X$ 
are set theoretical solutions of Yang-Baxter equation 
with a compatibility condition, we define an invariant  for virtual (or classical)
 knots/links using  non commutative 2-cocycles pairs  $(f,g)$ that generalizes the 
one defined in \cite{FG2}. We also define,
a group $\Uncfg=\Uncfg(X,S,\beta)$
 and functions $\pi_f, \pi_g\colon  X\times X\to \Uncfg(X)$ governing all 2-cocycles in $X$.
We exhibit examples of computations achieved using \cite{GAP2015}.
\end{abstract}

\section*{Introduction and preliminaries}

In \cite{FG2} we constructed an invariant for knots and links using 
noncommutative 2-cocycles, that is, for $(X,\sigma)$ a special solution of the Yang-Baxter
 equation (see definitions of biquandle below) and a map
$f:X\times X\to G$, where $G$ is a (eventually) non-abelian group, and $f$ satisfies
certain equations that we call noncommutative 2-cocycle conditions. In this way a noncommutative
version of the state-sum invariant can be defined. In this work we generalize 
this construction for virtual knots and links. Since a (diagram of a)
virtual link has two types of crossings, for a given set
$X$ of possible labels for the semi arcs, we  need two rules for coloring semi arcs in a crossing, say
$(X,S)$ and $(X,\beta)$, and also we 
need two types of weights. We consider
pairs $f,g:X\times X\to G$ that we call  ``noncommutative 2-cocycle pairs". The strategy is
to ask invariance under generalized (i.e. classical, virtual or mixed) Reidemeister moves
both for colorings and for products of weights in a given order.  
Next we also consider a universal group $\Uncfg$ that is 
the universal target of noncomutative 2-cocycle pairs for a given $(X,S,\beta)$. 
As a consequence of this construction, the invariant that {\em a priori} depends on the set of colorings and a choice of a non-commutative 2-cocycle pair
is actually determined by intrinsic properties of the set of colorings.
 
The contents of this work are as follows: after recalling the combinatorial definition
of a virtual link or knot, we introduce, in {\em Section 1},
 the notion of  non-abelian 2-cocycle pair.
Using this notion  we propose the noncommutative invariant 
in Definition \ref{defmain}, proving that it is actually an invariant. In {\em Section 2}
we define a group together with a noncommutative 2-cocycle pair
that has the universal property as target of noncommutative 2-cocycles. This group 
is defined in terms of generators and relations, and it is actually computable for 
virtual pairs
of small cardinality.  
We end by computing invariants of some virtual knots and links using this universal
group.

\begin{defi}\label{YBdefi}
A set theoretical solution of the Yang-Baxter equation is a pair
$(X,\sigma)$ where $\sigma:X\times X\to X\times X$ is a bijection satisfying
\[
( \id\times \sigma)( \sigma\times \id)( \id\times \sigma)
=( \sigma\times \id)( \id\times \sigma)( \sigma\times \id)
\]

Notation:  $\sigma(x,y)=(\sigma^1(x,y),\sigma^2(x,y))$ and
$\sigma^{-1}(x,y)=\overline{\sigma}(x,y)$.

A solution $(X,\sigma)$ is called
non degenerated, or {\em birack} if in addition: 
\begin{enumerate}
 \item
 ({\em left invertibility})
 for any $x,z\in X$ there exists a unique $y$ such that $\sigma^1\!(x,y)=z$, 
 \item ({\em right invertibility}) for any $y,t\in X$ there exists a unique $x$ such that $\sigma^2(x,y)=t$.
\end{enumerate}
A birack is called {\em biquandle} if, given $x_0\in X$, there exists a unique $y_0\in X$ such that
$\sigma(x_0,y_0)=(x_0,y_0)$. In other words, if there exists a bijective map $s:X\to X$ such
that
\[
\{(x,y):\sigma(x,y)=(x,y)\}=
\{(x,s(x)): x\in X\}
\]
See Lemma 0.3 in \cite{FG2} for biquandle equivalent conditions.
\end{defi}

Following Kauffman (see  \cite{K}), a virtual link or knot can be defined using 
diagrams with two
 types of crossings: classical and virtual ones; a  virtual crossings will be 
 a  4-valent vertex with a small circle around it. 
Virtual links/knots may be considered  to be equivalence classes of planar 
virtual knot diagrams under the
equivalence relation generated by  the
three (classical) Reidemeister moves,  the virtual moves and a mixed Reidemeister move.
\[
\begin{array}{ccc}
\includegraphics[scale=.2]{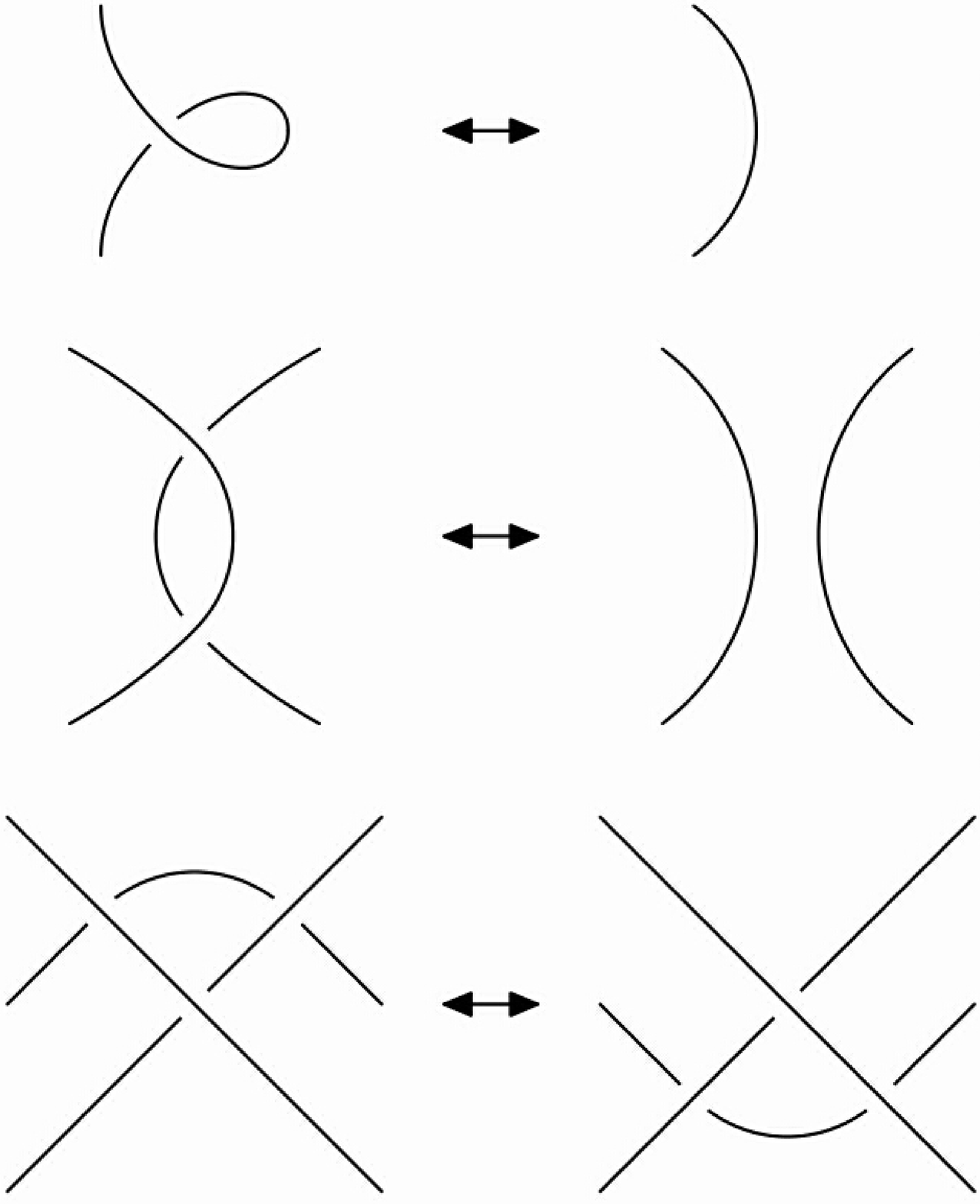}&&
\includegraphics[scale=.2]{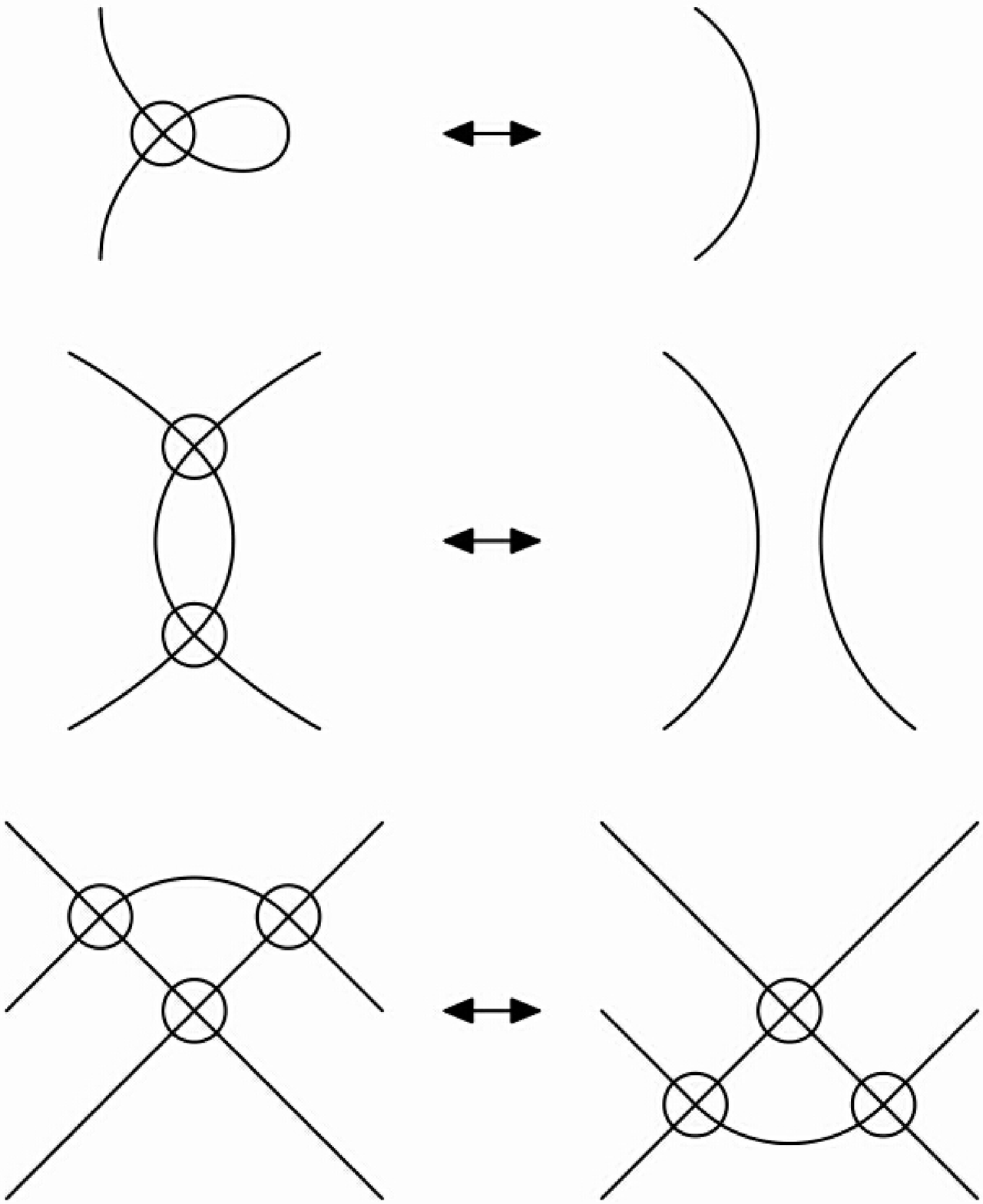}\\
\hbox{Classical Reidemeister moves:}&&\hbox{Virtual Reidemeister moves:}\\
\hbox{RI,RII and RIII.}&&\hbox{vRI, vRII, vRIII}
\end{array}
\] 

\[
\begin{array}{c}
\includegraphics[scale=.2]{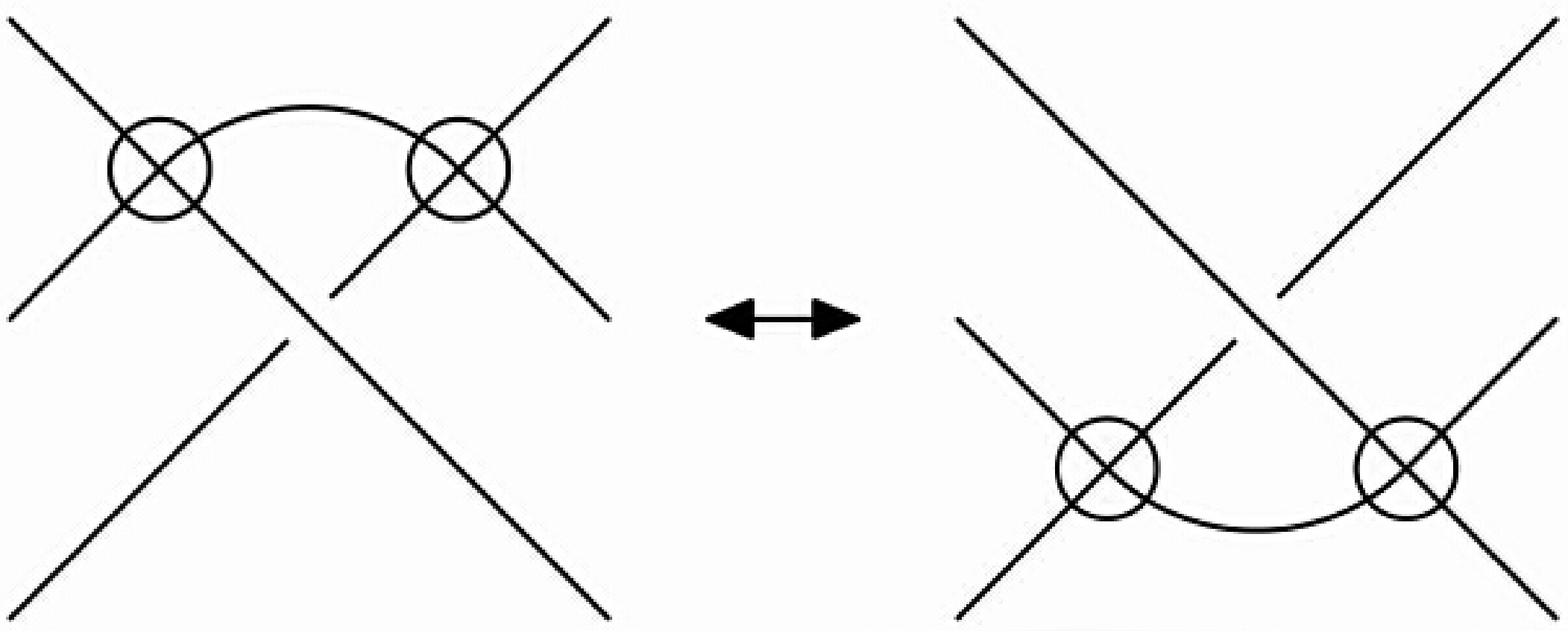}\\
\hbox{Mixed Reidemeister move: mixed RIII}
\end{array}
\]  
All links and knots considered
in this work will be oriented ones.
A useful reduction is proved in \cite{CN}:
 
\begin{lem}\label{eqmoves} (Lemma 2.4, \cite{CN}) The classical and virtual II moves, together with one oriented mixed RIII or vRIII move, imply the other oriented 
mixed RIII and vRIII moves. That is, we can reverse the direction of any strand in type mixed RIII or vRIII move using a sequence of RII and vRII
moves. 
 \end{lem}

\begin{defi}
A pair of biquandles $(X,S)$, $(X,\beta)$, (shortly $(X, S,\beta)$) is called a  {\em virtual pair} 
if $\beta^2=1$ and  
$(1\times \be)(S\times 1)(1\times \be)=(\be\times 1)(1\times S)(\be\times 1)$.
This notion is also called  virtual invariant in  \cite{BF}.
\end{defi}

\begin{ex}
If $(X,S)$ is a biquandle and $a\in\Aut(X,S)$, that is, $a:X\to X$ is a bijection satisfying
$(a\times a)S(a^{-1}\times a^{-1})=S$, then one can consider
$\beta(x,y)=(a^{-1}y,ax)$. It is easy to check that we get a virtual pair in 
that way.
\end{ex}

Not every virtual pair arise as in the above construction, if $S$ is  involutive (i.e. $S^2
=\id$) with $S(x,y)\neq (a^{-1}y,a x)$, then $(X,S,S)$ is a virtual pair.
But there are also different examples with non-involutive $S$,  already with  $|X|=3$. The
 following is an example with
 cardinal 4:

\begin{ex}\label{exZ4}
$
X=\Z/4\Z,\
S(x,y)=(-y,x+2y)
$
\[
\beta(x,y)=
\left\{
\begin{array}{cc}
(y,x)&\hbox{ if $x$ or $y$ is odd}\\
(y+2,x+2)&\hbox{ if $x$ and $y$ are even}
\end{array}
 \right.
\]

\end{ex}

\section{Non-abelian 2-cocycle pair}

We begin this section by introducing the notion of noncommutative 2-cocycle pair. 
If one analyzes the properties that a general weight (see subsection \ref{weights}) must satisfy in order to generalize
the construction  given in \cite{FG2}
to the virtual case, then one ends with the following definition:
 \begin{defi} 
 \label{nc2} Let  $H$ be  a (not necessarily abelian) group and $(X, S,\beta)$  a virtual pair. 
 A pair of  functions $f,g:X\times X\rightarrow H$
is a {\em  noncommutative 2-cocycle pair } if:
\begin{itemize}
\item the pair $f,S$ satisfies:
\begin{itemize}
\item[f1)] 
$ f\big(x,y\big)f\big(S^2(x,y),z\big)
 =f\big(x,S^1\!(y,z)\big)f\big(S^2(x,S^1\!(y,z)),S^2(y,z)\big)
$, 
\item[f2)]
$f\big(S^1\! (x, y), S^1\!(S^2(x,y),z)\big)=f\big(y,z\big)$,
\item[f3)]
$f(x, s(x))= 1$ (recall the map $s\colon X\to X$ from Definition \ref{YBdefi}), 
\end{itemize}
\item the pair $g,\be$ satisfies:
\begin{itemize}
\item[g1)]  $g(x, s_{\be}(x))=1$  (notice that $\beta$ involutive implies that 
$(X,\beta)$ is a biquandle, hence,  there is an associated map $s_\beta:X\to X$), 
\item[g2)] $g(x,y)g\big(\be(x,y)\big)=1$,
\item[g3)] 
$ g\big(x,y\big)g\big(\be^2(x,y),z\big)
 =g\big(x,\be^1\!(y,z)\big)g\big(\be^2(x,\be^1\!(y,z)),\be^2(y,z)\big)
$,
\item[g4)] 
$ g\big(y,z\big)g\big(\be^2(x,\be^1(y,z)),\be^2(y,z)\big)
 =g\big(x,y\big)g\big(\be^1(x,y),\be^1(\be^2(x,y),z)\big)
$,
\item[g5)] 
$ g\big(y,z\big)g\big(x,\be^1(y,z)\big)
 =g\big(\be^2\!(x,y),z\big)g\big(\be^1(x,y),\be^1(\be^2(x,y),z)\big)
$,

\end{itemize} 
\item and compatibility conditions between $f,g,\be,S$:
\begin{itemize}
 \item[m1)] $g\big(y,z\big)=g\big(S^1(x,y),\be^1(S^2(x,y),z)\big)$,
 \item[m2)] $g(y,z)g\big(x,\be^1(y,z)\big)=g\big(S^2(x,y),z\big)g\big(S^1(x,y),\be^1(S^2(x,y),z\big)$,
 \item[m3)] $g\big(x,\be^1(y,z)\big)f\big(\be^2(x,\be^1(y,z)),\be^2(y,z)\big)=
f(x,y)g\big(S^2(x,y),z\big)$

 \end{itemize}
 \end{itemize}
are satisfied for any $x, y,z \in X$.
\end{defi}

\begin{ex} Let $(X,S,\beta)$ be as in Example \ref{exZ4}, let
$H$ be the group with generators $\{a,b,c,d,h\}$ and relations
\[
bc=cb,\
c^2=1,\
[ h,a]=[h,b]=[h,c]=[h,d]=1,\] define $f,g:X\times X\to H$ by the tables
\[
\begin{array}{c||r|r|r|r}
f&0&1&2&3\\
\hline
\hline
0&1&a&1&a\\
\hline
1&b&c&bc&1\\
\hline
2&1&d&1&d\\
\hline
3&b&1&bc&c\\
\end{array}
\hskip 2cm 
\begin{array}{c || l|l|l|l}
g&0&1&2&3\\
\hline
\hline
0&1&h&1&h\\
\hline
1&h^{-1}&1&h^{-1}&1\\
\hline
2&1&h&1&h\\
\hline
3&h^{-1}&1&h^{-1}&1\\
\end{array}
\]
One can check by hand that the pair $(f,g)$ is a 2-cocycle pair,
 and after Theorem \ref{teouncfg}
we will see that any other 2-cocycle pair $(\wt f,\wt g):X\times X\to \wt H$ 
necessarily factorizes
 through this pair and a group homomorphism $\rho:H\to  \wt {H}$.
\end{ex}

\begin{ex}\label{exflip}
If $X=\{1,2\}$ and $S$=$\beta$=flip, then the cocycle  conditions f3, g1 and g2 are
\[
f(1,1)=f(2,2)=g(1,1)=g(2,2)=1\]
\[
g(1,2)=:h,\ g(2,1)=h^{-1}
\]
Call $a=f(1,2)$,  conditions f1, f2, g3-g5 and m1 m2 are trivially satisfied and condition m3 is simply
\[
ah=ha,\ bh=hb.\]
So, if one takes $H$ the group freely generated by $\{a,b,h\}$ with relations
$ah=ha$ and $bh=hb$, then the pair $f,g$ defined by $g(1,2)=h=g(2,1)^{-1}$,
$f(1,2)=a$, $f(2,1)=b$ and $f(1,1)=f(2,2)=g(1,1)=g(2,2)=1$ is a 2-cocycle pair.
\end{ex}

Next we consider some special cases and analyze the general equations for each case.
\subsubsection*{Some special cases}

If $(X,S)$ is a biquandle, $a\in\Aut(X,S)$ and $\beta(x,y)=(a^{-1}y,ax)$, then 
equations f1,f2,f3 remain the same. An easy computation shows that
g3 together with the choice of $y=ax$ gives $g(x,z)=g(ax,az)$ for all $x,z$.
Using this condition, the other equations may be simplified giving the following (equivalent)
set:
\begin{itemize}
\item[g0)]  $g(x,z)=g(ax,az)$ 
\item[g1)]  $g(x, ax)=1$, 
\item[g2)] $g(x,y)g\big(a^{-1}y,ax\big)=1$,
\item[g3)] 
$ g\big(x,y\big)g\big(x,z\big) =g\big(x,z\big)g\big(x,y\big)$,
\item[g4)] 
$ g\big(y,z\big)g\big(x,y\big) =g\big(x,y\big)g\big(y,z\big)$,
\item[g5)] $ g\big(y,z\big)g\big(x,z\big) =g\big(x,z\big)g\big(y,z\big)$,
\end{itemize}

\begin{itemize}
 \item[m1)] $g\big(y,z\big)=g\big(S^1(x,y),a^{-1}z\big)$,
 \item[m2)] $g(y,z)g\big(x,a^{-1}z\big)=g\big(S^2(x,y),z\big)g\big(y,z\big)$,
 \item[m3)] $g\big(x,a^{-1}z\big)f\big(ax,ay\big)=
f(x,y)g\big(S^2(x,y),z\big).$

 \end{itemize}

Notice that g3, g4, g5 are automatic if the group is abelian.
Actually, when the group is abelian, the equations m1, m2 and m3 can be replaced by
\begin{itemize}
 \item[m1)'] $g\big(y,z\big)=g\big(S^1(x,y),a^{-1}z\big)$,
 \item[m2)'] $g\big(x,z\big)=g\big(S^2(x,y),az\big)$,
 \item[m3)'] $f\big(ax,ay\big)=f(x,y).$

 \end{itemize}

Another interesting situation is when the biquandle is given by a quandle, that is 
$S(x,y)=(y, x\t y)$. In this case f2 is automatic, and one can make explicit $S^1$ and $S^2$ giving the following:
\begin{prop}
If $X=(Q,\t)$ is a quandle and  $a\in\Aut(Q)$, then $(f,g)$ is  a non-abelian
2-cocycle pair for $\beta(x,y)=(a^{-1}y,ax)$ and $S(x,y)=(y,x\t y)$
if and only if they verify the following equations
\begin{itemize}
\item[f1)] 
$ f\big(x,y\big)f\big(x\t y,z\big)
 =f\big(x,z\big)f\big(x\t z,y\t z\big)
$, 
\item[f3)]
$f(x, x)= 1$, 
\end{itemize}

\begin{itemize}
\item[g0-m1-m2)]  $g(x,z)=g(ax,z)=g(x,az)$, 
\item[g1)]  $g(x, x)=1$, 
\item[g2)] $g(x,y)g\big(y,x\big)=1$,
\item[g3)] $ g\big(x,y\big)g\big(x,z\big)
 =g\big(x,z\big)g\big(x,y\big)
$,
\item[g4)] 
$ g\big(y,z\big)g\big(x,y\big)
 =g\big(x,y\big)g\big(y,z\big)
$,
\item[g5)] 
$ g\big(y,z\big)g\big(x,z\big)
 =g\big(x,z\big)g\big(y,z \big)
$,
\end{itemize} 
\begin{itemize}
 \item[m3)] $g\big(x,z\big)f\big(ax,ay\big)=
f(x,y)g\big(x\t y,z\big)$.
 \end{itemize}

\end{prop}
\begin{coro}\label{coroinvo}
Let $(Q,\t)$ be a quandle and assume $a\in Aut(Q,\t)$ 
is a maximum cycle (e.g. $X=\Z_n$ and $a(x)=x+1$). If
$(f,g)$ is  a non-abelian
2-cocycle pair for $\beta(x,y)=(a^{-1}y,ax)$ and $S(x,y)=(y,x\t y)$, then $g\equiv 1$.
\end{coro}
\begin{proof} Given $x,z\in X$, since
$g(x,z)=g(ax,z)$ we have
$g(x,z)=g(a^nx,z)$ for all $n$. If one assumes that the action of $a$ in $X$ is transitive then
we have $a^nx=z$ for some $n$ and so
$g(x,z)=g(z,z)=1$.

\end{proof}

Notice that m3 is nontrivial even if $g\equiv 1$. We mention another special case:
\begin{coro}\label{coroinvo2}
Let $X=\{1,\dots,n\}$,  $a=(1,2,\dots,n)$, then
$(f,g)$ is  a non-abelian
2-cocycle pair for $S$=flip and $\beta(x,y)=(a^{-1}y,ax)$ 
if and only if $g\equiv 1 $,
\[
f(x,y)=f(ax,ay),\
f(x,x)=1\]
and
\[
 f\big(x,y\big)f\big(x\t y,z\big)
 =f\big(x,z\big)f\big(x\t z,y\t z\big).
\]
In particular, for $n=2$, $f$ is fully determined by $f(1,2)$.
\end{coro}


\subsubsection*{Particular case $\beta$=flip and $H$ an abelian group}
The specialization in this case gives the equations
f1, f2, f3 together with
\begin{itemize}
\item[g1)]  $g(x,x)=1$, 
\item[g2)] $g(x,y)g\big(y,x\big)=1$.
\end{itemize} 
 
\begin{itemize}
 \item[m1)] $g\big(y,z\big)=g\big(S^1(x,y),z\big)$,
 \item[m2)] $g\big(x,z\big)=g\big(S^2(x,y),z\big)$,
 \end{itemize}

\subsubsection*{Connected components}

Let $(Q,\t)$ be a quandle and consider the
equivalence relation generated by $x\t y\sim x \ \forall x,y\in Q$.
Recall that $Q$ is called {\em connected} if there is only one equivalence class.

Generalizing this definition, one can consider, for a biquandle $(X,S)$, the
equivalence relation generated by
 $$\forall x, y \in X, \ \ x\sim S^1(x,y)\hbox{ and }y\sim S^2(x,y),$$
that is, if $S(x,y)=(y',x')$ then $x\sim x'$ and $y\sim y'$. 
The equivalence classes are called {\em connected components}, and 
the biquandle
$(X,S)$ is called {\em connected} if there is only one class.
Clearly if $S$
is given by a quandle then this definition agrees with the previous one.

For a virtual pair $(X,S,\beta)$ there is also a natural equivalence relation, the 
one  generated by
 $$\forall x, y \in X, \ \ x\sim S^1(x,y)\sim \beta^1(x,y)$$ and $$y\sim S^2(x,y)\sim \beta^2(x,y).$$
That is, if $S(x,y)=(y',x')$ and $\beta(x,y)=(y'',x'')$ we are  setting
$x\sim x'\sim x''$, $y\sim y'\sim y''$.
 \begin{defi}\label{defconnected}
For a virtual pair  $(X,S,\beta)$, equivalent classes of elements of $X$ are called
connected components. The
 virtual pair  $(X,S,\beta)$ is called connected if there is only one
 class.
 \end{defi}

\begin{rem}
If one is interested in {\em knots}, then it is clear that
one can restrict the attention to connected virtual pairs, because a coloring
of a knot only uses elements of the same connected component of $X$.
 \end{rem}

\begin{ex}\label{ex24}
If the biquandle $(X,S)$ is already connected then $(X,S,\beta)$ is obviously connected, 
the same for the biquandle $(X,\beta)$. With the help of a computer
 one can check that for cardinal 2,3,5 these are the only cases. For cardinal 4 there are
examples of connencted virtual pairs $(X,S,\beta)$ with nonconnected $(X,S)$ and
nonconnected $(X,\beta)$. More precisely, there are 167 (isomorphism classes of) connected
virtual pairs of size 4,  and 10 of them have disconnected $S$  and $\beta$.
Similar thing happens in cardinal 6, see table in subsection \ref{small}.

\end{ex}

A straightforward consequence of m3 is the following corollary:

\begin{coro}
If $(X,S)$ is  a connected biquandle  and $\beta$=flip then $g\equiv 1$.
\end{coro}

On the opposite side, if the biquandle is trivial (i.e.  $S(x,y)=(y,x)$) then
m1) and m2) are trivial, the conditions for $g$ are only $g(x,x)=1$ and $g(x,y)=g(y,x)^{-1}$,
as in Example \ref{exflip}.

Next we will construct an invariant for oriented knots or links from a virtual pair 
$(X,S,\beta)$ and a 2-cocycle pair $(f,g)$.

\subsection{Weights}\label{weights}
Let $(X,S,\be)$ be a virtual pair, $H$ a group and  $f,g:X\times X\to H$ a non-abelian 2-cocycle pair.
Let $L=K_1\cup\dots\cup K_r$ be a virtual oriented link
diagram on the plane, where $K_1,\dots , K_r$ 
are connected components, for some positive integer $r$. 
A {\em coloring} of $L$ by $X$ is a rule that assigns an element of $X$ to each semi-arc of $L$, in such a way that
for every  regular crossing (figure on the left corresponds to a positive crossing and figure on the right to a negative one):
\[
\xymatrix{
x\ar@{->}[rd]|\hole&y\ar[ld]\\
z&t
}
\hskip 2cm 
\xymatrix{
z\ar@{->}[rd]&t\ar[ld]|\hole\\
x&y
}
\]
where $(z,t)=S(x,y)$ and  in case of a virtual crossing:

\[
\xymatrix{
x\ar@{->}[rd]|\bigotimes  &y\ar[ld]\\
z&t
}
\] where $(z,t)=\beta(x,y)$.

\begin{rem}
The conditions for $(X,S,\beta)$ to be a virtual pair are precisely the 
compatibility of the set of colorings with the Reidemeister moves
(RI, RII, RIII, vRI, vRII, vRIII  and mixed RIII), so given $(X,S,\beta)$ a virtual pair,
the number of colorings of a link (or a knot) using $(X,S,\beta)$ is an invariant
of that link (or knot).
\end{rem}

Let $\mathcal{C}\in Col_X(L)$ be a coloring of $L$ by $X$ and
 $(b_1,\dots, b_r)$ a set of base points on the components $(K_1,\dots, K_r)$.
Let $\tau^{(i)}=\{\tau_1^{(i)},\dots,\tau_{k_{(i)}}^{(i)}\}$, for $i=1,\dots, r$, be the ordered set of  regular crossings such that the under-arc
belongs to component $i$ or it is virtual corssing involving component $i$.
The order of the set  $\tau^{(i)}$ is given by the orientation of the component starting at the base point.

At a positive crossing $\tau$, let $x_{\tau}, y_{\tau}$ be the color  on the incoming arcs. 
The {\it Boltzmann weight} at a positive crossing $\tau$ is 
$B_{f,g}(\tau, \mathcal{C})=f(x_{\tau}, y_{\tau})$.
At a negative crossing $\tau$, denote $S(x_{\tau}, y_{\tau})$  the colors  on the incoming  arcs.
The {\it Boltzmann weight} at $\tau$ is 
$B_{f,g}(\tau, \mathcal{C})=f(x_{\tau},y_{\tau})^{-1}$
\[
\xymatrix@-1pc{
x_{\tau}\ar@{->}[rdd]|\hole&y_{\tau}\ar[ldd]\\
&\leadsto  f(x_{\tau},y_{\tau}	)\\
S^1\!(x_{\tau}, y_{\tau})&S^2(x_{\tau},y_{\tau})
}
\hskip 0.5cm
\xymatrix@-1pc{
S^1\!(x_{\tau}, y _{\tau})\ar@{->}[rdd]&S^2(x_{\tau},y_{\tau})\ar[ldd]|\hole\\
&\leadsto f(x_{\tau},y_{\tau})^{-1}\\
x_{\tau}&y_{\tau}
}
\]
At a virtual crossing $\tau$, let $x_{\tau}, y_{\tau}$ be the color  on the incoming arcs. 
The {\it Boltzmann weight} at $\tau$ is 
$B_{f,g}(\tau, \mathcal{C})=g(x_{\tau}, y_{\tau})$.
\[
\xymatrix@-1pc{
x_{\tau}\ar[rdd]|\bigotimes &y_{\tau}\ar[ldd]\\
&\leadsto g(x_{\tau}, y_{\tau})&\\
\beta^1(x_{\tau},y_{\tau})&\beta^2(x_{\tau},y_{\tau})\\
}
\]

We will show that a convenient product of these weights is invariant under Reidemeister moves. More precisely, take an oriented component, start 
at a base point, take the product of  Boltzmann weights associated to the crossing whenever
 it is a virtual crossing, or the crossing is classical but one is going through the under arc.

For a group element $h\in H$, denote $[h]$  the conjugacy class to which $h$ belongs.
\begin{defn}\label{defmain}
The set of conjugacy classes  
 \[
  \overrightarrow{\Psi}(L,f,g)= \overrightarrow{\Psi}_{(X,f.g)}(L)
  = \{[\Psi_i(L,\mathcal{C},f,g )]\}_{\underset{\mathcal{C}\in Col_X(L)}{1\leq i\leq r}}
 \]
where  $
        \Psi_i(L,\mathcal{C},f)=\prod^{k(i)}_{j=1}B_{fg}(\tau^{(i)}_j, \mathcal{C})
       $
(the order in this product is following the orientation of the component)
is called the   {\em conjugacy biquandle cocycle invariant} of the link.   
\end{defn}

The following is our main theorem:

\begin{teo}\label{invariante}
The   conjugacy biquandle cocycle     $\Psi$ is well defined and then define a knot/link
 invariant.
\end{teo}
\begin{rem} This invariant clearly generalizes the one constructed in \cite{FG2} 
by simply taking $\beta=flip$ and $g\equiv 1$. On the opposite side,
 if  one chooses $f\equiv 1$ and general $g$, this invariant will be trivial on classical 
links or knots, so a nontrivial $g$ may detect virtuality.
\end{rem}

\begin{ex} 
Take the group $H=\langle h\rangle$, $X=\{1,2\}$, $S=\beta=\hbox{flip}$, $f\equiv 1$,
$g(1,1)=g(2,2)=1$, $g(1,2)=g(2,1)^{-1}=h$. Here we show all possible colorings and the corresponding invariants.
\[
\xymatrix{
&1\ar@/^0.5pc/[rd]|(.4)\hole&1\ar@/_0.5pc/[ld]\\
&1\ar@/_0.5pc/[rd]|(.65)\bigotimes&1\ar@/^0.6pc/[ld]&\hskip 1pc\leadsto \{1,1\}\\
&1\ar@/^3pc/[uu]&1\ar@/_3pc/[uu]
}\hskip 2cm
\xymatrix{
&1\ar@/^0.5pc/[rd]|(.4)\hole&2\ar@/_0.5pc/[ld]\\
&2\ar@/_0.5pc/[rd]|(.65)\bigotimes&1\ar@/^0.6pc/[ld]&\hskip 1pc\leadsto \{h^{-1}, h^{-1}\}\\
&1\ar@/^3pc/[uu]&2\ar@/_3pc/[uu]
}\]
\[
\xymatrix{
&2\ar@/^0.5pc/[rd]|(.4)\hole&1\ar@/_0.5pc/[ld]\\
&1\ar@/_0.5pc/[rd]|(.65)\bigotimes&2\ar@/^0.6pc/[ld]&\hskip 1pc\leadsto \{h, h\}\\
&2\ar@/^3pc/[uu]&1\ar@/_3pc/[uu]
}\hskip 2cm
\xymatrix{
&2\ar@/^0.5pc/[rd]|(.4)\hole&2\ar@/_0.5pc/[ld]\\
&2\ar@/_0.5pc/[rd]|(.65)\bigotimes&2\ar@/^0.6pc/[ld]&\hskip 1pc\leadsto \{1,1\}\\
&2\ar@/^3pc/[uu]&2\ar@/_3pc/[uu]
}
\]
In particular, this link is nontrivial and  non classical.
\end{ex}

\begin{proof}(of Theorem \ref{invariante}).
We will check the product of weights is invariant under Reidemeister moves.
In \cite{FG2} calculations due to regular crossings can be found, remains to consider virtual and mixed Reidemeister moves.  
Following Lemma \ref{eqmoves} we will check only one orientation of 
arcs in each Reidemeister move
(the rest will be equivalent). 
\begin{itemize}
              \item Virtual Reidemeister type I move: 
              
 $\beta(x,s_{\beta}(x))=(x,s_{\beta}(x))$ 
\vskip -1.5cm
\[    \xymatrix{
&    &\\
x \ar@{-} `d[r] `r[ru]|(.4)\bigoplus  `u[rl]_{s_{\beta}(x)}`l[rd]`d[dr] &&&x \ar@{-}@/_1pc/[rd]\\
& x&&&x
    }
  \]
   the condition (g1) $g(x, s_{\beta}(x))=1$, assures
that the factor due to this crossing
will not change the product.

\item Virtual Reidemeister type II move:

%
Take, for example, the following diagram:
\[
\xymatrix{
x\ar@/^0.6pc/[rd]|(.45)\bigotimes &y\ar@/_0.6pc/[ld]&& x\ar@/^0.7pc/[dd]&\ar@/_0.7pc/[dd] y\\
\be^1 (x,y)\ar@/_0.6pc/[rd]|(.55)\bigotimes &\be^2(x,y)\ar@/^0.6pc/[ld]&& &\\
x&y&& x&y\\
}\]
%
%
%
%
%
%
%
%
%

Condition (g2) assures the product of weights due to these crossings will not change the product.

 \item Virtual Reidemeister type III move:

Start by naming the incoming arcs $x,y,z$, then the outcoming  arcs are respectively equal as $\beta$ is  a solution of YBeq.
\[
  \xymatrix{
  &y\ar@{->}[rddd]|\bigotimes&\ar@{->}[lddd]z&\\
x\ar@{->}[rrr]|(.4)\bigotimes|(.52)\bigotimes&&&{}^{\be^2(\be^2(x,y),z)}\\  
&
&&&&\\
&{}^{\be^1(\be^1(x,y),\be^1(\be^2(x,y),z))}&
{}^{\be^2(\be^1(x,y),\be^1(\be^2(x,y),z))}&\\} 
\]
\[\xymatrix{
  &y\ar@{->}[rddd]|\bigotimes&\ar@{->}[lddd]z&\\
&&&\\
x\ar@{->}[rrr]|(.33)\bigotimes|(.42)\bigotimes&&&{}^{\be^2(\be^2(x,\be^1(y,z)),\be^2(y,z))}\\  
&{}^{\be^1(x,\be^1(y,z))}&{}^{\be^1(\be^2(x,\be^1(y,z)),\be^2(y,z))}&} 
\]
The product of the weights following the horizontal arc, in the first diagram, is: 
\[
A_1=g(x, y)g(\be^2(x,y),z)
\]
and in the second diagram is: 
\[
 B_1=g(x, \be^1(y,z))g(\be^2(x,\be^1(y,z)),\be^2(y,z))
\]

$A_1=B_1$ is  item (g3) in Definition \ref{nc2}.

The product of the weights following the arc labeled by $y$, in the first diagram, is: 
\[A_2=g(x, y)g(\be^1(x,y),\be^1(\be^2(x,y),z))\] 
and in the second, is: 
\[
 B_2=g(y, z)g(\be^2(x,\be^1(y,z)),\be^2(y,z))
\]

$A_2=B_2$ is item (g4) in Definition \ref{nc2}.  

The product of the weights following the arc labeled by $z$, in the first diagram, is: 
\[A_3=g(\be^2(x,y), z)g(\be^1(x,y),\be^1(\be^2(x,y),z))\] 
and in the second, is: 
\[
 B_3=g(y, z)g(x,\be^1(y,z))
\]

$A_3=B_3$ is  item (g5) in Definition \ref{nc2}.

 \item Mixed virtual Reidemeister type III move:

Start by naming, in both diagrams,  $x,y,z$ the incoming arcs. The outcoming  arcs are respectively equal as $(X,S,\beta)$ is  a virtual pair.
\[
  \xymatrix{
  &y\ar@{->}[rddd]|\bigotimes&\ar@{->}[lddd]z&\\
x\ar@{->}[rrr]|(.41)\hole|(.525)\bigotimes&&&{}^{\be^2(S^2(x,y),z)}\\  
&
&&&&\\
&{}^{\be^1(S^1(x,y),\be^1(S^2(x,y),z))}&
{}^{\be^2(S^1(x,y),\be^1(S^2(x,y),z))}&\\} 
\]
\[\xymatrix{
  &y\ar@{->}[rddd]|\bigotimes&\ar@{->}[lddd]z&\\
&&&\\
x\ar@{->}[rrr]|(.34)\bigotimes|(.43)\hole&&&{}^{S^2(\be^2(x,\be^1(y,z)),\be^2(y,z))}\\  
&{}^{\be^1(x,\be^1(y,z))}&{}^{S^1(\be^2(x,\be^1(y,z)),\be^2(y,z))}&} 
\] 

The product of the weights following the arc labeled by $x$ in the first diagram is:

\[
A_1=f(x,y)g\big(S^2(x,y),z\big)
\]
and in the second diagram is: 
\[
 B_1=g\big(x,\be^1(y,z)\big)f\big(\be^2(x,\be^1(y,z)),\be^2(y,z)\big)
 \]

$A_1=B_1$ is   item (m3) in Definition \ref{nc2}.

The product of the weights following the arc labeled by $y$ in the first diagram is:

\[
A_1=g\big(S^1(x,y),\be^1(S^2(x,y),z)\big)
\]
and in the second diagram is: 
\[
 B_1=g(y,z)
 \]

$A_1=B_1$ is  item (m1) in Definition \ref{nc2}.

The product of the weights following the arc labeled by $z$ in the first diagram is:

\[
A_1=g\big(S^2(x,y),z\big)g\big(S^1(x,y),\be^1(S^2(x,y),z)\big)
\]
and in the second diagram is: 
\[
 B_1=g(y,z)g\big(x,\be^1(y,z)\big)
 \]

$A_1=B_1$ is   item (m2) in Definition \ref{nc2}.

\end{itemize}

 This shows that the product of the weights does not change  under generalized Reidemeister moves.  A change of base points
 causes cyclic permutations of Boltzmann weights, and hence the invariant is defined up to conjugacy.

 \end{proof}

\begin{ex}
In \cite{GPV} the authors mention that there are several ways to generalize the notion
of linking number to the virtual case. For 2-component  links,   they give
 two independent versions of the linking
number: the invariant $lk_{\frac{1}{2}}$
may be computed as a sum of signs of real crossings where
the first component passes over the second one. Similarly, 
$lk_{\frac{2}{1}}$ is defined by
exchanging the components in the definition of $lk_{\frac{1}{2}}$.

In our context, previous definitions can be achieved in the following way: take a 
two component (virtual) link. Take $(X,S,\beta)$ the
 virtual pair with $S=\beta=flip$ and  $f,g$ a 2-cocycle pair with $g=1$. 
Take two different elements  $1,2\in X$. Color 
``the first'' component with color $1$ and ``the second'' component with color $2$. The invariant for the {\em second} component will be
$f^{lk_{\frac{1}{2}}}(2,1)$.
The invariant for the first component will be $f^{lk_{\frac{2}{1}}}(1,2)$. Recall 
(see Example \ref{exflip}) that for
$X$=$\{1,2\}$, $S$=$\beta$=flip , cocycle pairs can be obtained 
considering
$G=Free\{a,b\}\times Free\{h\}$ and $f,g:X\times X\to G$ defined by
\[
f(1,1)=f(2,2)=g(1,1)=g(2,2)=1\]
\[
g(1,2)=h,\ g(2,1)=h^{-1}
\]
\[
f(1,2)=a,\ f(2,1)=b
\]
\end{ex}

\begin{ex}
Take $X=\{0,1\}=\Z/2\Z$,
$S=\hbox{flip}$ and $\beta$ given by
\[
\beta(0,0)=(1,1),\ \beta(1,1)=(0,0),\]\[
\beta(0,1)=(0,1),\ \beta(1,0)=(1,0)
\]
One can check that this rule can be written as $\beta(x,y)=(y-1,x+1)$ so  it is an
involutive biquandle,
and also one can easily check that the coloring rule for $(X,S,\beta)$ 
is the rule of ``changing the color when going trough a virtual crossing
and not changing the color when the crossing is classical'', just 
as in  \cite{Kself}. If one considers the 2-cocycle equations then
(see Corollary \ref{coroinvo2})
 we are lead to  $g\equiv 1$
and a group $H=\langle a \rangle$ with  $f:X\times X\to H$
satisfying
\[
f(1,0)=f(0,1)=a,\ f(0,0)=1,\ f(1,1)=1.
\]
If one uses this  cocycle pair for a classical 2-component link, then  exponent of $a$ is
 the  linking number. So, if one uses this cocycle pair for virtual knots or links,
one gets a different generalization of the linking
number to the virtual case (see \cite{Kself} for the notion of ``self-linking number").
\end{ex}

\begin{rem}\label{remflipinvo}
Given $(X,S=flip)$, the condition for an involutive $\beta$ to be compatible with $S$, 
in the sense that $(X,S,\beta)$ is a virtual pair, 
is non-trivial. Nevertheless, there are plenty of examples; for instance, if $|X|=7$, there are
3456 involutive solutions,  1959 of them are compatible with $S=flip$.
\end{rem}
 
%
%

\subsection{Cohomologous pairs}
From  the following lemma we propose the notion of cohomologous 2-cocycle pair:

\begin{lem}\label{cohom}
Let $f,g:X\times X\to G$ be a 2-cocycle pair and $\lambda:X\to G$ be a map.
If one defines
\[
f_\lambda(x,y):=\lambda(x)f(x,y)\lambda(S^2(x,y))^{-1}
\]
\[
g_\lambda(x,y):=\lambda(x)g(x,y)\lambda(\beta^2(x,y))^{-1}
\]
then 
\begin{itemize}
\item $f_\lambda$ always satisfies f1,
\item $f_\lambda$ satisfies f2 $\iff$  $\lambda(y)=\lambda(S^1(x,y))$ for all $x$,
\item $f_\lambda$ satisfies f3 $\iff$  $\lambda(x)=\lambda(s_S(x))$ for all $x$,
\item $g_\lambda$ satisfies g1 $\iff$  $\lambda(y)=\lambda(s_\beta(x,y))$ for all $x$,
\item $g_\lambda$ always satisfies g3,
\item  $\lambda(y)=\lambda(\be^1(x,y))$ for all $x,y$ $\iff$
 $\lambda(\be^2(x,y))=\lambda(x)$ for all $x,y$.
\item If  $g(x,y)\equiv 1$, then $1_\lambda$ verifies
g2 $\iff$ $\lambda(x)\lambda(\beta^2(x,y))^{-1}
\lambda(\beta^1(x,y))\lambda(y)^{-1}=1$ $\forall x,y$.
\item If $\lambda(y)=\lambda(\beta^1(x,y))$ $\forall x,y$, then
$g_\lambda$ satisfies g2 $\iff$
$[\lambda(x),g(x,y)][g(x,y),\lambda(y)]=1$ for all $x,y$, where the brackets denote the commutator. If also $\lambda(x)$ commutes with $g(x,y)$ for all $y$
then $g_\lambda=g$.
\end{itemize}
\end{lem}

\begin{defi}
Let $H$ be a group,   $(X,S,\beta)$ be a virtual pair. Two 2-cocycle pairs
$(f,g)$ and $(\tilde{f},\wt g)$  are called cohomologous if
$g=\wt g$ and  there exists $\lambda:X\rightarrow H$ such that 
\[\tilde{f}(x,y)=\lambda(x)f(x,y)[\lambda(S^2(x,y))]^{-1}\]
with $\lambda$ satisfying
\begin{itemize}
 \item $\lambda(x)=\lambda(s_S(x))$,
 \item $\lambda(y)=\lambda(S^1(x,y))$,
 \item $\lambda(y)=\lambda(\beta^1(x,y))$,
\item for all $x$, and $y$,  $\lambda(x)$ commutes with $g(x,y)$.
 \end{itemize}
\end{defi}
 
From  Lemma \ref{cohom} above one can easily prove the following:
\begin{prop}
If $(f,g)$ is a 2-cocycle pair and 
$(\wt f,g)$ is cohomologous to $(f,g)$ then $(\wt f,g)$ is also a 2-cocycle pair.
\end{prop}
And one can also prove the expected result:

 \begin{prop}\label{cohomologous} 
 If  $(f, g), (\wt f,g)$ are two  cohomologous noncommutative 2-cocycle  pairs
 then  $$[\Psi_i(L,\mathcal{C},f,g)]
 =[\Psi_i(L,\mathcal{C},\wt f, g)].$$ \end{prop}
\begin{proof} Let us suppose $\wt f(x,y)=\gamma(x) f(x,y)[\gamma(S^2(x,y))]^{-1}$.
Take a link $L$, pick a connected component $K$ and a base point. 
If every crossing in $K$ is virtual it is obvious.
If every crossing in $K$ is classical see \cite{FG2}.
If $K$ has both, virtual and classical crossings: 

 \[
   \xymatrix{
& y \ar@{->}[dd]&& z \ar@{->}[dd]&\\
x\ar@{->}[rrrr]|(0.2)\hole|(.625)\bigoplus^{S^2(x,y)}
&&&&\beta^2(S^2(x,y),z)\\
&S^1(x,y)&&\beta^1(S^2(x,y),z)&
 }\]  
 the product of weights for the horizontal line is:
$$\wt f(x,y)g(S^2(x,y),z)=\lambda(x)f(x,y)[\lambda(S^2(x,y))]^{-1}g(S^2(x,y),z)=$$
$$\lambda(x)f(x,y)g(S^2(x,y),z)\big[\lambda((S^2(x,y))\big]^{-1}=\lambda(x)f(x,y)g(S^2(x,y),z)\big[\lambda(\beta^2(S^2(x,y),z))\big]^{-1}$$
 
  \[
    \xymatrix{
& y \ar@{->}[dd]&& z \ar@{->}[dd]&\\
x\ar@{->}[rrrr]|(0.195)\bigoplus|(0.62)\hole^{\beta^2(x,y)}&&&&S^2(\beta^2(x,y),z)\\
&\beta^1(x,y)&&S^1(\beta^2(x,y),z)&
 }\]
 the product of weights for the horizontal line is:
 \[g(x,y)\wt f(\beta^2(x,y),z)
=g(x,y)\lambda(\beta^2(x,y))f(\beta^2(x,y),z)\big[\lambda\big(S^2(\beta^2(x,y),z)\big)\big]^{-1}
\]
hence
\[
g(x,y)\lambda(x)f(\beta^2(x,y),z)\big[\lambda\big(S^2(\beta^2(x,y),z)\big)\big]^{-1}
\]\[=
\lambda(x)g(x,y)f(\beta^2(x,y),z)\big[\lambda\big(S^2(\beta^2(x,y),z)\big)\big]^{-1}\]
 
  \end{proof}

 \section{Universal noncommutative 2-cocycle pair}

Given a virtual pair $(X,S,\beta)$ we shall define a  group together with a
universal 2-cocycle pair in the following way:

\begin{defi}\label{def:unc}
Let $\Uncfg=\Uncfg(X,S,\beta)$ be the
 group freely generated by symbols
$(x,y)_f$   and $(x,y)_g$ with relations
\begin{itemize}
\item[f1)] 
$\big(x,y\big)_{\!\! f}   \ \big(S^2(x,y),z\big)_{\!\! f}
 =\big(x,S^1\!(y,z)\big)_{\!\! f}\ \big(S^2(x,S^1\!(y,z)),S^2(y,z)\big)_{\!\! f} $ 
\item[f2)]
$\big(S^1\! (x, y), S^1\!(S^2(x,y),z)\big)_{\!\! f}\ =\big(y,z\big)_{\!\! f} $
\item[f3)]
$\big(x, s(x)\big)_{\!\! f} = 1$ 
\item[g1)]  $\big(x, s_{\be} (x)\big)_{\!\! g} =1$ 
\item[g2)] $\big(x,y\big)_{\!\! g}\ \big(\be(x,y)\big)_{\!\! g}\ =1$
\item[g3)] 
$\big (x,y\big)_{\!\! g}\ \big(\be^2(x,y),z\big)_{\!\! g}\ 
 =\big(x,\be^1\!(y,z)\big)_{\!\! g}\ \big(\be^2(x,\be^1\!(y,z)),\be^2(y,z)\big)_{\!\! g} $
\item[g4)] 
$ \big(y,z\big)_{\!\! g}\ \big(\be^2(x,\be^1(y,z)),\be^2(y,z)\big)_{\!\! g}\ 
 =\big(x,y\big)_{\!\! g}\ \big(\be^1(x,y),\be^1(\be^2(x,y),z)\big)_{\!\! g}$
\item[g5)] 
$ \big(y,z\big)_{\!\! g}\ \big(x,\be^1(y,z)\big)_{\!\! g}\ 
 =\big(\be^2\!(x,y),z\big)_{\!\! g}\ \big(\be^1(x,y),\be^1(\be^2(x,y),z)\big)_{\!\! g}$
 \item[m1)] $\big(y,z\big)_{\!\! g}\ =\big(S^1(x,y),\be^1(S^2(x,y),z)\big)_{\!\! g} $
 \item[m2)] $\big(y,z\big)_{\!\! g}\ \big(x,\be^1(y,z)\big)_{\!\! g}\ =\big(S^2(x,y),z\big)_{\!\! g}\ \big(S^1(x,y),\be^1(S^2(x,y),z\big)_{\!\! g} $
 \item[m3)] $\big(x,\be^1(y,z)\big)_{\!\! g} \ \big(\be^2(x,\be^1(y,z)),\be^2(y,z)\big)_{\!\! f}=
\big(x,y\big)_{\!\! f}\ \big(S^2(x,y),z\big)_{\!\! g}$.
 \end{itemize}
Denote $f_{xy}$ anf $g_{xy}$ the class  in $\Uncfg$
of $(x,y)_{\! f}$ and $(x,y)_{\! g}$ respectively. 
We also define $\pi_f,\pi_g:X\times X\to \Uncfg$ by
\[
\pi_f,\pi_g\colon X\times X\to \Uncfg\]\[
\pi_f(x,y):=f_{xy},\]\[
\pi_g(x,y):=g_{xy}
\]
\end{defi}

The following is immediate from the definitions:
\begin{teo}\label{teouncfg}
Let $(X,S,\beta)$ be virtual pair:
\begin{itemize}
\item   The pair of maps
$\pi_f,\pi_g\colon X\times X\to \Uncfg$
is a noncommutative 2-cocycle pair.
\item 
Let $H$ be  a group and
 $f,g:X\times X\to H$ a  noncommutative 2-cocycle pair, then there exists a unique group homomorphism
$\rho :\Uncfg\to H$ such that
$f=\rho \circ \pi_f$
and $g=\rho\circ \pi_g$
 \[
 \xymatrix{
 X\times X\ar[d]_{\pi_f}\ar[r]^f&H\\
 \Uncfg\ar@{-->}[ru]_{ \rho}
 }\hskip 1cm
 \xymatrix{
 X\times X\ar[d]_{\pi_g}\ar[r]^g&H\\
 \Uncfg\ar@{-->}[ru]_{ \rho}
 }\]
       \end{itemize}
\end{teo}


\begin{rem} $\Uncfg$ is functorial. That is, 
 if $\phi:(X,S,\beta)\to (Y,S',\beta')$ is a morphism of 
virtual pairs, namely $\phi$
satisfy
\[
(\phi\times\phi)S(x_1,x_2)=
S'(\phi x_1,\phi x_2),
\hskip 1cm
(\phi\times\phi)\beta(x_1,x_2)=
\beta'(\phi x_1,\phi x_2)
\]
then, $\phi$ induces a (unique) group homomorphism
$\Uncfg(X)\to \Uncfg(Y)$
satisfying
\[
f_{x_1x_2}\mapsto
f_{\phi x_1 \phi x_2}
\hskip 1cm 
g_{x_1x_2}\mapsto
g_{\phi x_1 \phi x_2}
\]
\end{rem}
\begin{proof}
One needs to prove that the assignment $f_{x_1x_2}\mapsto
f_{\phi x_1 \phi x_2}$ and $g_{x_1x_2}\mapsto
g_{\phi x_1 \phi x_2}$
are compatible with the relations defining $\Uncfg(X)$ and $\Uncfg(Y)$ respectively,
and this is clear since $(\phi\times\phi)\circ S=S\circ (\phi\times\phi)$
and
 $(\phi\times\phi)\circ\beta=\beta\circ (\phi\times\phi)$.
\end{proof}

\begin{rem}
 In order to produce an invariant of a knot or link, given a solution $(X,S,\beta)$,
 we need to produce a coloring
 of the knot/link by $X$, and then find a noncommutative 2-cocycle, but since $\Uncfg$ 
is functorial,
 given $X$ we always have the universal 2-cocycle pair  $f,g:X\times X\to \Uncfg(X)$,
 and hence, the information given by the invariant was already
included in the combinatoric of the colorings.

 Also, if $\phi:X\to X$ is a bijection commuting with $S$ and $\beta$,
 then, given a coloring and its invariant calculated
 with the universal cocycle, we may apply $\phi$
 to each color and get another coloring, and this will produce the same invariant pushed by $\phi$ in $\Uncfg$.
 
\end{rem}
 
\begin{ex}\label{unc1} Computations of Example \ref{exflip}
show that for $X=\{1,2\}$ and $S=\beta=flip$, $\Uncfg(X)\cong Free(a,b)\times Free(h)$
where $(1,1)_f=(2,2)_f=(1,1)_g=(2,2)_g=1$, $(1,2)_f=a$, $(2,1)_f=b$, $(1,2)_g=h$,
$(2,1)_g=h^{-1}$.
\end{ex}

\begin{ex}\label{unc2} If $X=\{1,2\}$, $S(x,y)=(y+1,x+1)$ (mod 2) and  $\beta=flip$, then
\[
 \Uncfg(X)\cong Free(c)\]
where $c=(1,1)_f=(2,2)_f$,  $1=(1,2)_f=(2,1)_f$, and $(x,y)_g=1$ for all $x,y\in X$.
This virtual pair does not give the same information as the previous example (since
for instance $g\equiv 1$), but it gives a different way to generalize the linking number to virtual links.
\end{ex}

\subsection{Some examples of virtual pairs of small cardinality\label{small}}

Using GAP, the list of biquandles and involutive solutions, one can easily 
compute the list of (isomorphism classes of) virtual pairs of small cardinality.
We show the total amount of them in the following table. 
The amount grows very fast, for cardinal 6 the computer takes too long
to compute all virtual pairs, so we put on the table only parcial cases for $n=6$.
The notation
$(S,i_a)$ is for virtual pairs with biquandle $S$ and involutive $\beta$ of the
form $\beta(x,y)=(a^{-1}y,ax)$, with $a\in \Aut(X,S)$.
 Notice that for each $S$ there are
as many isomorphism classes of pairs $(S,i_a)$ as conjugacy classes
of $\Aut(X,S)$. 

\

\begin{tabular}{|c|c|c|c|c|c|}
\hline
n&2&3&4&5&6\\
\hline
all virtual pairs&4&90&3517&46658&\\
\hline
virtual pairs $(S,i_a)$&4&38&325&41278&111151\\
\hline
connected virtual pairs&3&26&167&138&836\\
\hline
conn. virtual pairs with&&&&&\\
non conn. $S$ and non conn $\beta$&0&0&10&0&84\\
\hline
\end{tabular}

\

The complete list in each case can be found in: \newline
http://mate.dm.uba.ar/\~{}mfarinat/papers/GAP/virtual

\section{Some virtual   knots/links and their n.c. invariants}

We begin with an example of colorings:
 \begin{ex}
Let $S$ be the dihedral quandle, that is $X=\{1,2,3\}$ and $S(x,y)=(y,x\t y)$ where
$x\t y=2y-x$ (mod 3). $\Aut(X,S)$ can be identified with the dihedral group
 $D_3=S_3$. There are three conjugacy classes in $D_3$, a set of representatives is
 $\{\id, (2,3),(1,2,3)\}$. For $a\in\Aut(X,S)$ denote $i_a$ the involutive biquandle given by 
$i_a(x,y)=(a^{-1}(y),a(x))$. In the following table we write the number of colorings 
of the Kishino's knots using the corresponding virtual pair, so we see that they are all 
different.
\[
\begin{array}{ccc}
 \includegraphics[scale=.15]{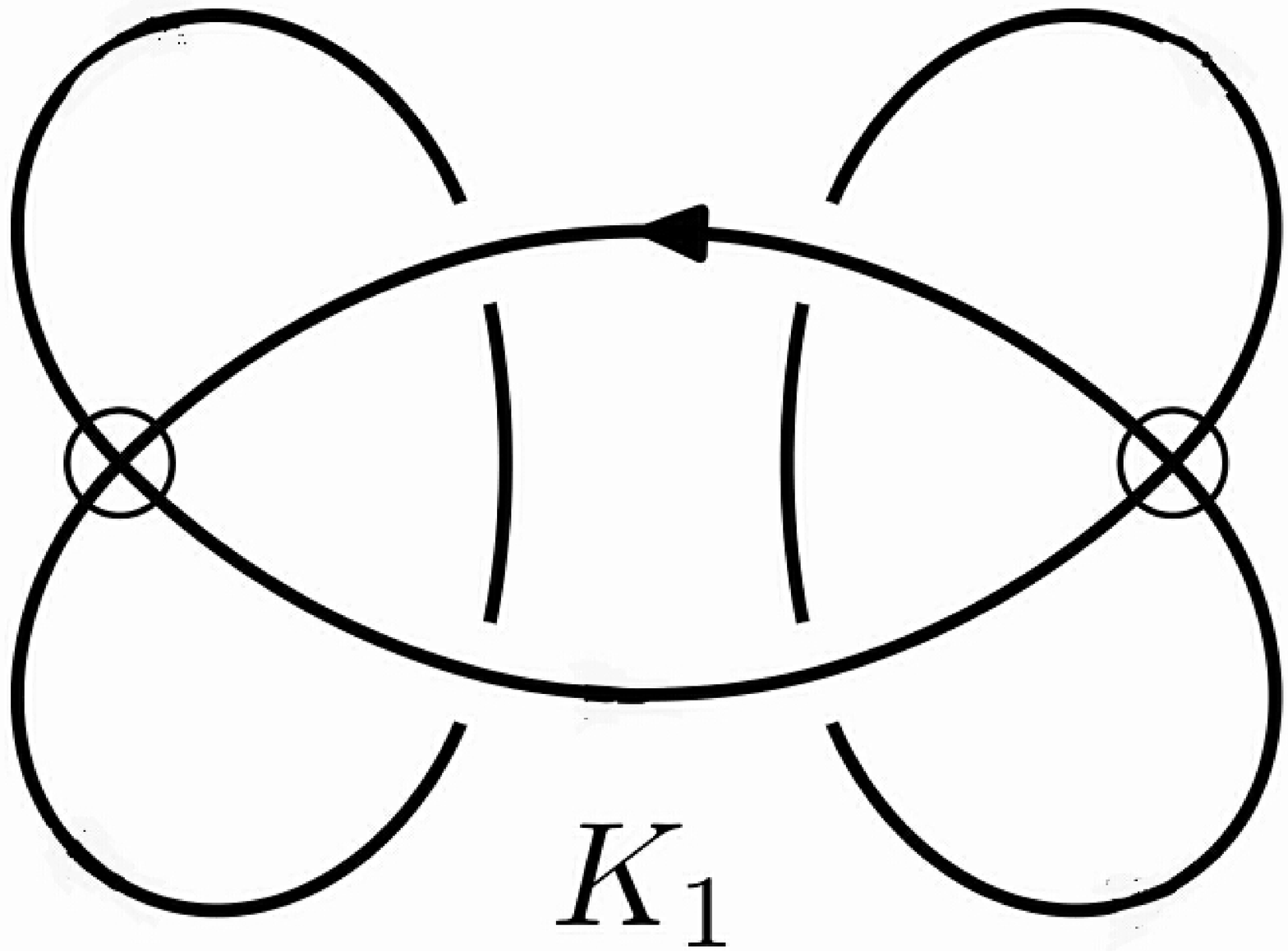}\hspace{2.5mm}&\hspace{1.5mm}\includegraphics[scale=.15]{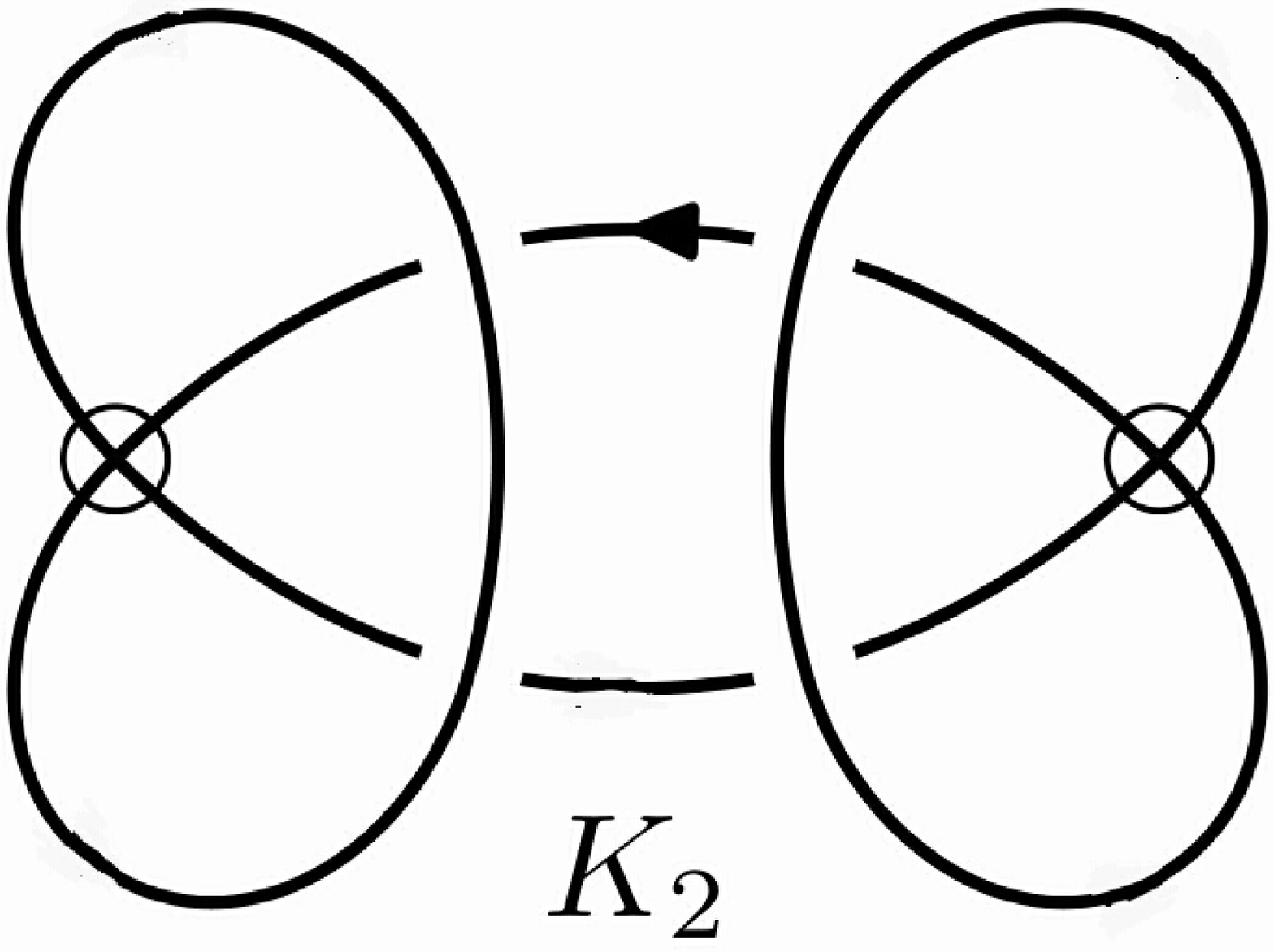}\hspace{1.5mm}&\includegraphics[scale=.15]{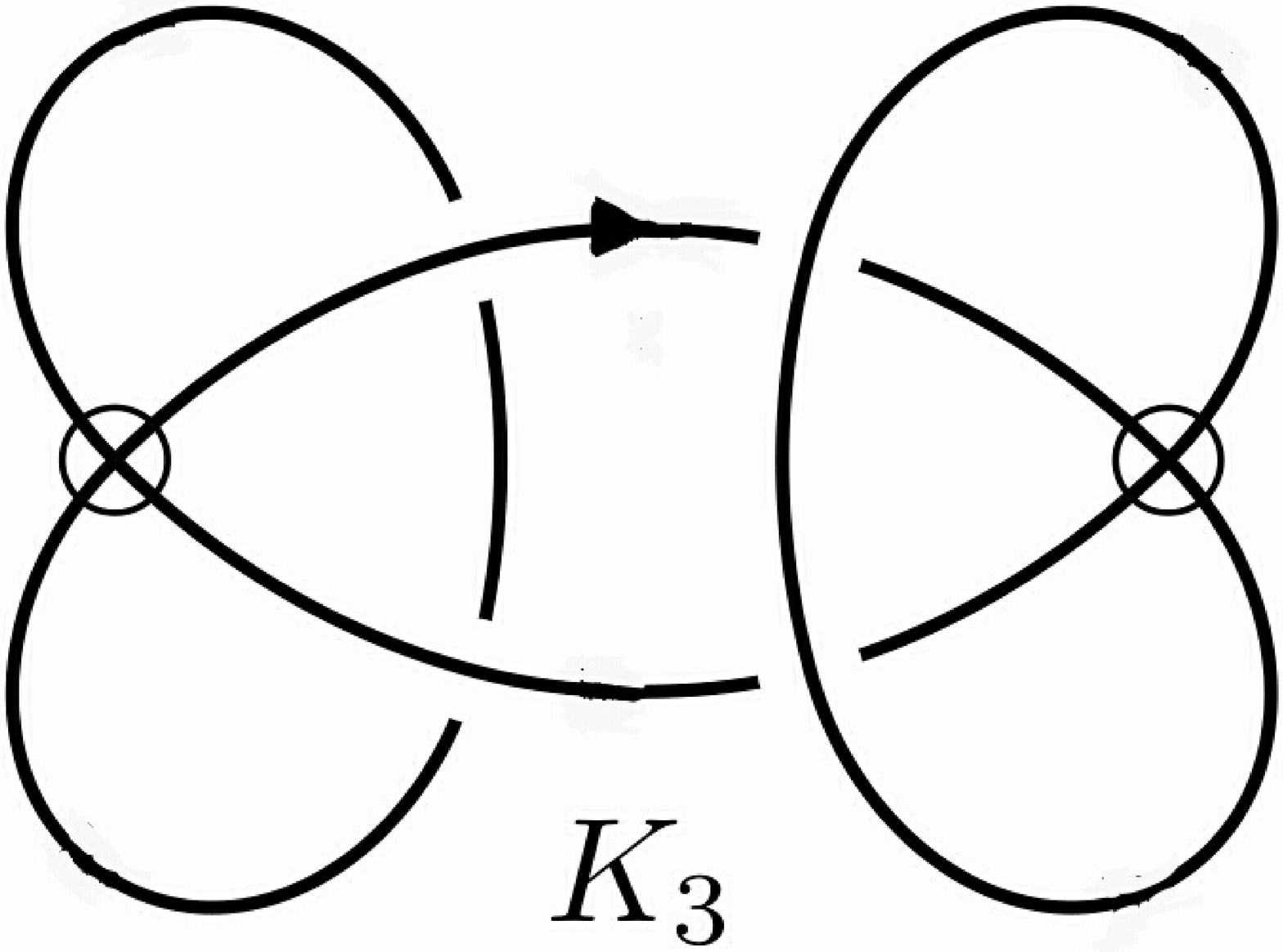}
\end{array}
\]

\[
\begin{array}{c||c|c|c}
\#colorings&(S,i_{id})&(S,i_{(2,3)})&(S,i_{(1,2,3)})\\
\hline\hline
K1&9&3&9\\
\hline
K2&3&9&3\\
\hline
K3&3&3&3
\end{array}
\]
Moreover, for $X=\{1,2,3,4\}$ with $S$ given by
\[\begin{array}{cccc}
S(1,1)=(1,1)& S(1,2)=(2,4)& S(1,3)=(4,2)& S(1,4)=(3,3) \\
S(2,1)=(3,4)& S(2,2)=(4,1)& S(2,3)=(2,3)& S(2,4)=(1,2) \\
S(3,1)=(4,3)& S(3,2)=(3,2)& S(3,3)=(1,4)& S(3,4)=(2,1) \\
S(4,1)=(2,2)& S(4,2)=(1,3)& S(4,3)=(3,1)& S(4,4)=(4,4)
\end{array}
\]
and $\beta=flip$, then the number of colorings of K3 is 16, so K3 is also nontrivial.
\end{ex}

\subsection{Links}

It is worth to notice that \cite{BF} computes virtual pairs of small cardinality. In that work,
some classes of virtual pairs are considered, the so called essential pairs, and the welded pairs.
Recall that there are ``forbidden'' Reidemeister moves:

\[
\begin{array}{cc}
 \includegraphics[scale=.35]{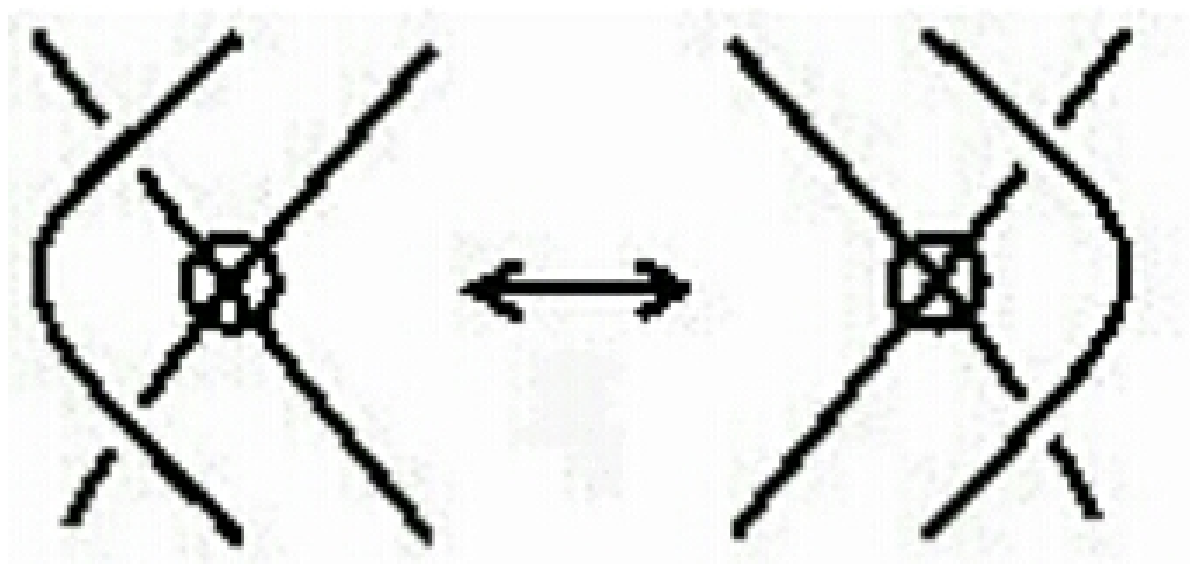}\hspace{2.5mm}&\hspace{1.5mm}\includegraphics[scale=.35]{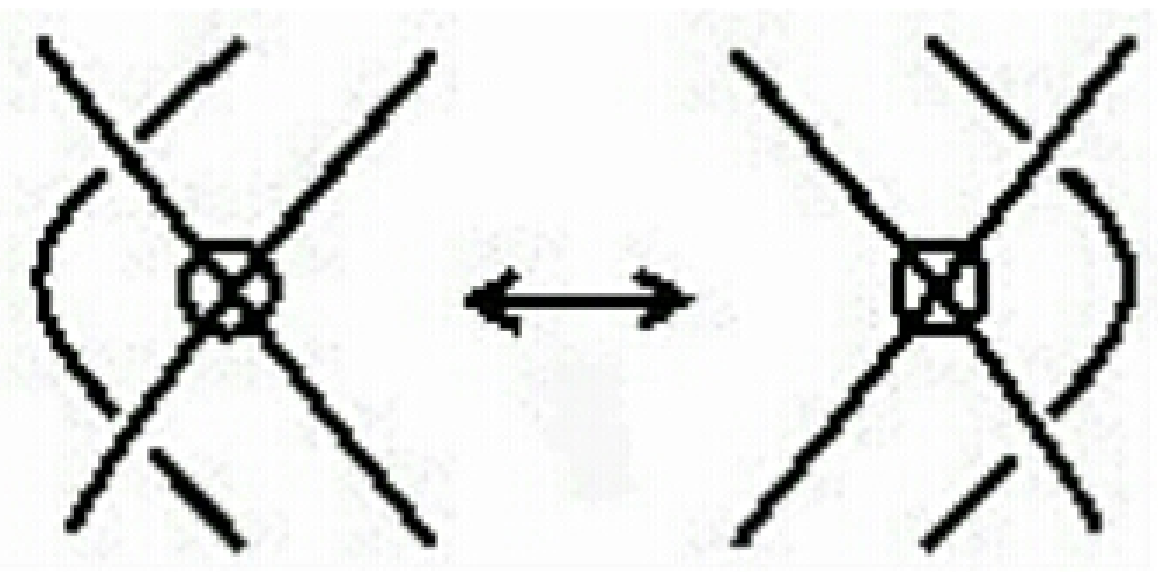}
\end{array}
\]

These moves are not allowed in virtual knots, and if one uses both forbidden moves, then
one can ``unknot" every knot/link. So, essential virtual pairs are pairs that {\em do not }
satisfy
those forbidden moves, in welded pairs a forbidden move is allowed (see \cite{BF}
for details). In this work we consider {\em all} virtual pairs, that's why we have more
virtual pairs that in \cite{BF}. in particular, for n=2,
the trivial example $(flip,flip)$ is not considered in \cite{BF}, 
and one can easily see that the number of colorings doesn't give any interesting information, 
just if the link is connected or not, but
the 2-cocycle invariant is highly nontrivial, as we show next.

From the list of 51 virtual links provided by A. Bartholomew 
(these are 2-component links with 4, 5 or 6 crossings), coloring with $(X=\{1,2\}, S=\beta=flip)$ (see Example \ref{unc1}) and computing the invariant 
(coloring will not  distinguish these links),
leaves 18 classes. To refine this, color with
 $(X=\{1,2\}, S=antiflip, \beta=flip)$ (i.e. Example \ref{unc2}) 
and compute the invariant. Using both invariants leaves 38 classes.
Furthermore, color with $X=\{1,2,3,4\}$ and all possible virtual pairs (without computing the invariant) and get 47 classes.

We exhibit three examples of links from this list, and their invariants:

\begin{ex}
\[
\begin{array}{ccc}
 \includegraphics[scale=.15]{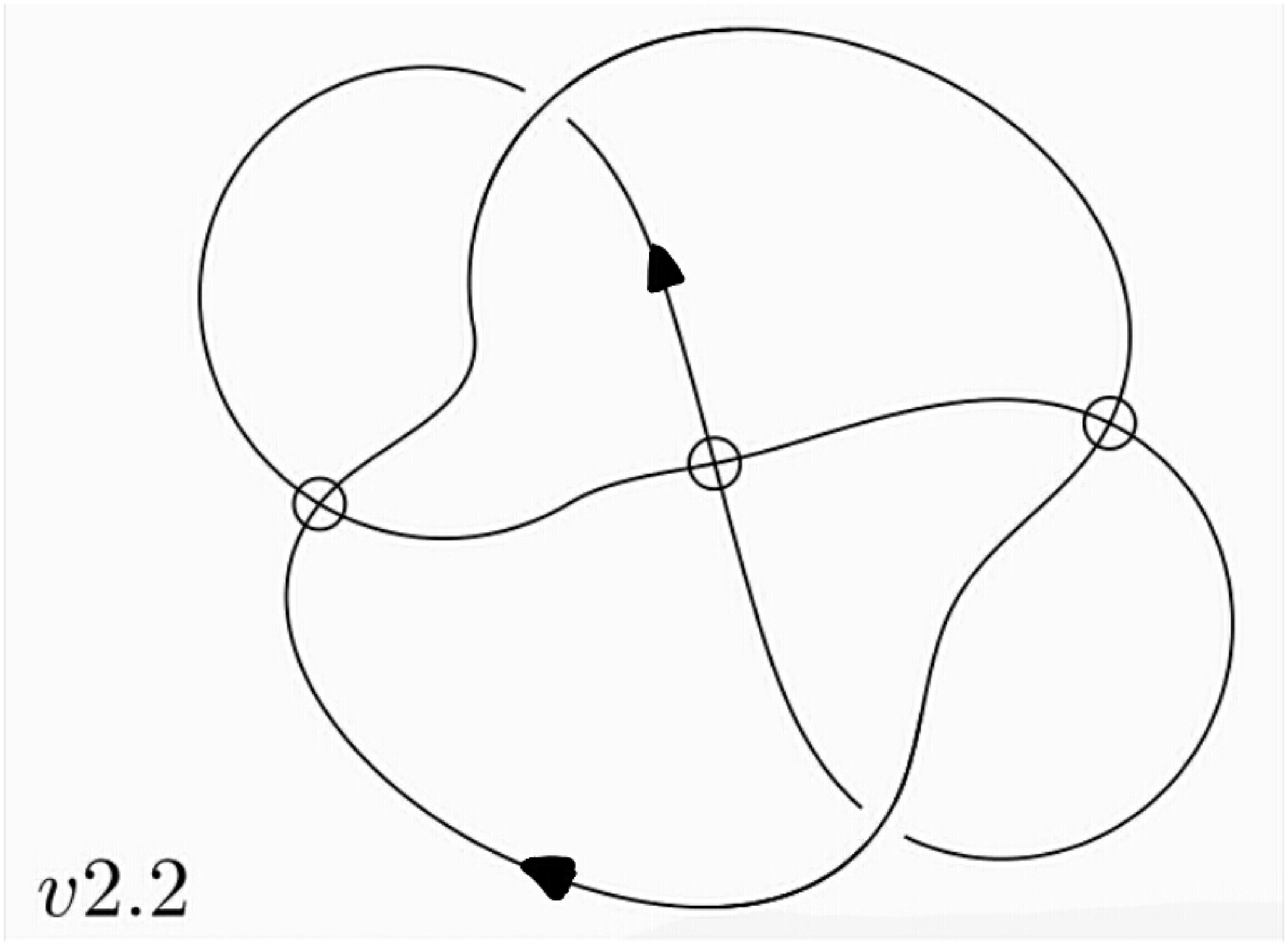}\hspace{2.5mm}&\hspace{1.5mm}\includegraphics[scale=.145]{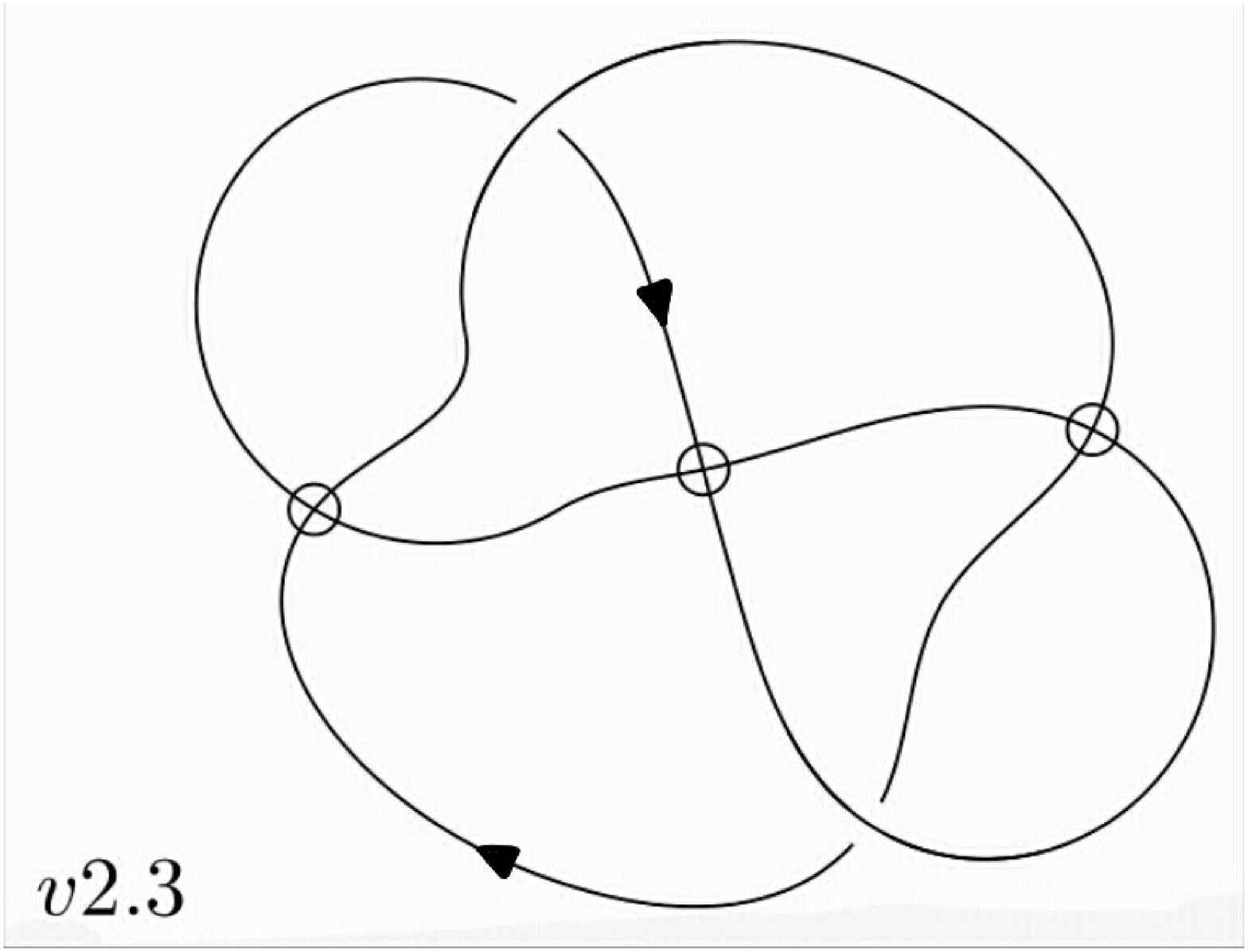}\hspace{1.5mm}&\includegraphics[scale=.168]{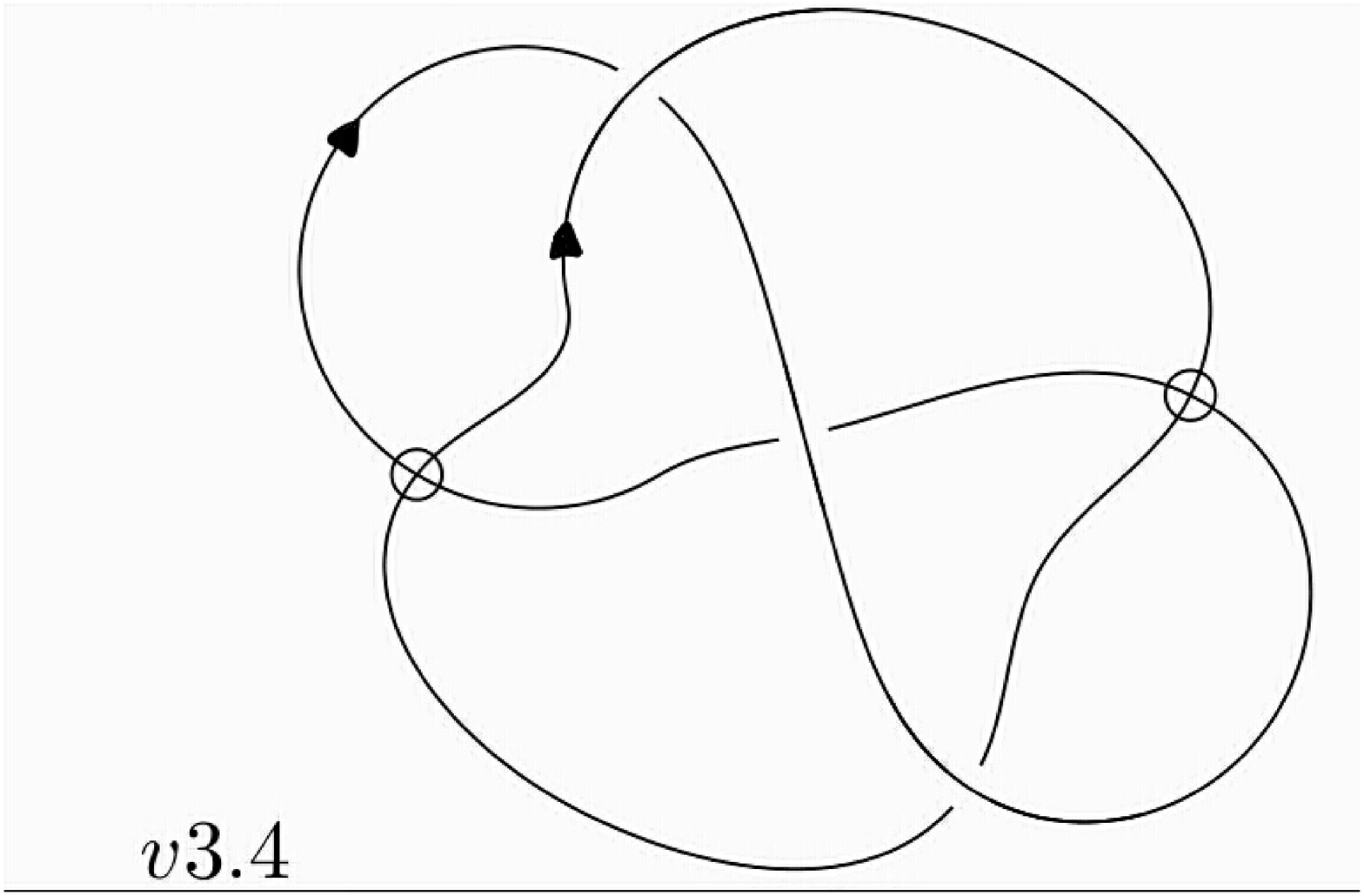}
\end{array}
\]
\[\begin{array}{|c||c|c|c|c|}
\hline
(X,S,\beta)&\Uncfg&v2.2&v2.3&v3.4\\
\hline\hline
(\{1,2\}\hbox{,flip,flip})&\langle a,b\rangle \times \langle h\rangle &4\{1,1\}&2\{a^{-1},b^{-1}\},&2\{a^{-1},b^{-1}\},\\
& &&2\{1,1\}&2\{1,1\}\\
\hline
(\{1,2\}\hbox{,a-flip,flip})&\langle c\rangle&4\{1,1\}&2\{c^{-2},1 \},&2\{c^{-1},c^{-1}\}, \\
& &&2\{1,1\}&2\{1,1\}\\
\hline
\end{array}
\]

Using $X=\{1,2\}$ and $S=\beta=flip$ (see Example \ref{unc1}) and computing the invariant gives 4 colorings to each $v2.2,v2.3$ and $v3.4$.  For every coloring the invariant of
$v2.3$ and $v3.4$
 gives $\{a^{-1}, b^{-1}\}$ twice and $\{1,1\}$ twice, but the same computation for $v2.2$  gives always $\{1,1\}$.
 
 To be able to distinguish   
 $v2.3$ from  $v3.4$, consider same set $X$ but $S=antiflip$ and $\beta=flip$ (see Example \ref{unc2}), again there are 4 possible colorings for each link.
 The invariant gives:
     $\{c^{-1},c^{-1}\}$ twice and $\{1,1\}$ twice for $v2.3$ and 
     $\{c^{-2},1\}$ twice and $\{c^{-1},c^{-1}\}$ twice for $v3.4$,
and always $\{1,1\}$ for $v2.2$.
                     \end{ex}

\begin{rem}
The exponent of $a$ (or $b$) is the first generalization of 
linking number to the virtual case (in the sense that if the link is classical then it gives the
 linking number). Using the second virtual pair, the 
exponent of $c$ is a different generalization of the linking number.
\end{rem}

\subsection*{A non-commutative example}

Consider $S$ given by the quandle $\{1,2,3,4\}$
 with operation
\[
-\t 1=\t  \ 2=(3,4),\]
\[ 
-\t 3=\t\ 4=(1,2)\]
that is, $S(x,y)=(y,x\t y)$, and $\beta$ the involutive solution
      
$
\beta(x,y)= ( l_x(y),r_y(x))$ where
\[
l_1=l_2=(1,2),\]
\[l_3=l_4=(1,2)(3,4),\]
 and $r_i=l_i$ ($i=1,2,3,4$).
(This is the 
pair number 248 in the list vp4 in \cite{FG}.)

Using the relations of $\Uncfg=\Uncfg(X,S,\beta)$ one can easily see that
\[
\begin{array}{rcccccccccccc}
1&=&(1,2)_f&=&(2,1)_f&=&(3,4)_f&=&(4,3)_f&=&(1,1)_g&=&(1,2)_g
\\
&=&(2,1)_g&=&(2,2)_g&=&(3,3)_g&=&(3,4)_g&=&(4,3)_g&=&(4,4)_g
\\
a&:=&(1,1)_f&=&(2,2)_f,
&&b&:=&(1,3)_f&=&(2,4)_f\\
c&:=&(1,4)_f&=&(2,3)_f,
&&d&:=&(3,1)_f&=&(4,2)_f\\
e&:=&(3,2)_f&=&(4,1)_f,
&&f&:=&(3,3)_f&=&(4,4)_f\\
&&h &:=&(1,3)_g&=&(1,4)_g&=&(2,3)_g&=&(2,4)_g\\
&&h^{-1}&=&(3,1)_g&=&(3,2)_g&=&(4,1)_g&=&(4,2)_g
\end{array}
\]
So, we have 7 generators, and if one (or a computer)  writes the list of all relations
in terms of $a,b,c,d,e,f,h$, one gets
\[b=ac,\
b=ca,\ c=ab, \ c=ba,\]
\[
ab=ba,\
ac=ca,\
ah=ha,\
bc=cb,\]
\[bh=hc,\
ch=hb,\
\]
\[
d=ef,\
d=fe,\ e=df,\
e=fd,\
dd=ee,\]
\[df=fd,\
dh=he,\
ef=fe,\]
\[eh=hd,\
fh=hf.\]
If one solves $b$, and $d$ in terms of $a,c,e,f,h$, equations above translate into

\[b=ac,\
   d=ef,\]
\[
ac=ca,\ c=aac, \ c=aca,\
aac=aca,\
\]
\[ah=ha,\
acc=cac,\
ach=hc,\]
\[
ch=hac,\
ch^{-1}=h^{-1}ac,
\]
\[
ef=fe,\ e=eff,\
e=fef,\
efef=ee,\]
\[eff=fef,\
efh=he,\
\]
\[eh=hef,\
fh=hf.\
\]
One can easily see that $a^ 2=1$, $f^2=1$, $a=[h,c]$ (=$hch^{-1}c^{-1}$), $f=[e^{-1},h]$.

\begin{rem}
Let $G$ be a group, $a,c,h\in G$ and assume
\[a^2=1,\ a=[h,c],\  [h,a]=[c,a]=1,\]
then $[h,c]=[c,h]=[c^{-1},h^{-1}]=[c^{-1},h]=[c,h^{-1}]$.
\end{rem}
Using this remark, it is an easy exercise to check the following characterization:

\begin{coro}Denote $a:=[h,c]$ and $f:=[h,e]$, then
\[
\Uncfg(S,\beta)\cong 
\frac{Free(h,c,e)}
{\left\langle\begin{array}{c}
a^2=[a,c]=[a,h]=1,
\\f^2=[f,e]=[f,h]=1\end{array}\right\rangle}.\]
\end{coro}

\begin{rem}
The element $a$ is nontrivial in $\Uncfg$.
\end{rem}
\begin{proof}
If one adds the relation $e=1$ then (recall $a:=[h,c]$) 
\[
G:=\Uncfg/\langle e=1\rangle\cong 
\frac{Free(h,c)}
{\left\langle
\ a^2=[a,c]=[a,h]=1,
\right\rangle}\]
$G$ can be described as a central extension of 
$\Z^2\cong Free(h,c)/\langle[h,c]\rangle$ over $\Z/2\Z\cong\langle a : a^2=1\rangle$,
 more precisely,
consider the set of monomials
\[
M:=\{h^ic^ja^\epsilon : i,j\in\Z, \epsilon=0,1\}
\]
then $M$ is a group with multiplication given by
\[
(h^ic^ja^\epsilon )(h^kc^la^\sigma)=
h^{i+k}c^{j+l}a^{jk+\epsilon+\sigma Mod \ 2}
\]
and clearly $S\cong M$, so $a\neq 1$ in $G$.
\end{proof}
 \begin{rem}
 A famous quotient of $\Uncfg$ is the quaternion group
$H=\{\pm1,\pm i,\pm j,\pm k\}$ where
 $e\mapsto 1$, $h\mapsto i$, $c\mapsto j$. One can see that relations go to 1,
 so we have a well-defined group homomorphism, and $a\mapsto -1$.
  \end{rem}
  
  \begin{rem}
  If one uses the abelianization of $\Uncfg$, then one gets essentially a (Laurent) polynomial in the variables $h$, $c,$ $e$, and clearly
  the element $a$ is trivial in  $(\Uncfg)_{ab}$, since $a=[h,c]$. But there are examples where the full non-commutative invariant gives $a$ as
  answer (see next example),
   so this non-commutative invariant refines the
   2-cocycle one with values only in commutative groups.
   \end{rem}

 \begin{ex}
 If one uses this virtual pair and the universal 2-cocycle, then the invariant for 
 the virtual link v2.3
 is $(a,a)$ twice, 
  $(f,f)$ twice, and 4 times $(1,1)$.
  \end{ex}

\subsection{State sum}

If the target group $(A,\cdot)$ is abelian, then one can perform the state-sum for a pair
of maps  $f,g:X\times X\to A$, defining Bolzman Weights in the same way. For a given
 coloring, consider the product over all crossings of the corresponding weights, and
 then sum over all colorings.
If one asks for Reidemeister invariance in this construction, the set of equations are:

\begin{itemize}
\item[ss-f1)] 
$f(x, s(x))= 1$, 
\item[ss-f2)]
$ f\big(x,y\big)f\big(S^2(x,y),z\big)f\big(S^1\! (x, y), S^1\!(S^2(x,y),z)\big)
=\\
 =f\big(x,S^1\!(y,z)\big)f\big(S^2(x,S^1\!(y,z)),S^2(y,z)\big)f\big(y,z\big)
$,
\item[ss-g1)]  $g(x, s_{\be}(x))=1$, 
\item[ss-g2)] $g(x,y)g\big(\be(x,y)\big)=1$,

\item[ss-g3)] 
$ 
g\big(x,y\big)
g\big(\be^2(x,y),z\big)
g\big(\be^1(x,y),\be^1(\be^2(x,y),z)\big)
= \\
=g\big(x,\be^1\!(y,z)\big)
g\big(\be^2(x,\be^1(y,z)),\be^2(y,z)\big)
g\big(y,z\big)
$
 \item[ss-m)] $ g(y,z)
 g\big(x,\be^1(y,z)\big)
f\big(\be^2(x,\be^1(y,z)),\be^2(y,z)\big)
 =\\
 \\=
 g\big(S^1(x,y),\be^1(S^2(x,y),z)\big)
g\big(S^2(x,y),z\big)
f(x,y)$.
 \end{itemize}
Conditions  ss-f1 and ss-f2 are a consequence of f1, f3, f3. Also ss-m 
follows from m1, m2, m3. We have ss-g1 and ss-g2 are the same as g1 and g2. But
g3, g4, and g5 imply (ss-g3)${}^2$, that is, assuming g3 g4 and g5 one can conclude
\[
g\big(x,y\big)^2
g\big(\be^2(x,y),z\big)^2
g\big(\be^1(x,y),\be^1(\be^2(x,y),z)\big)^2
=\]
\[=g\big(x,\be^1\!(y,z)\big)^2
g\big(\be^2(x,\be^1(y,z)),\be^2(y,z)\big)^2
g\big(y,z\big)^2
\] 
If the abelian group $A$ has no elements of order 2,
then a non-commutative 2-cocycle pair is also a commutative 2-cocycle. One can think of
 the group $(\Uncfg)_{ab}$ as a nontrivial way of producing cocycles for 
virtual state-sum invariants, at least when
$(\Uncfg)_{ab}$ has no elements of order 2.

\section{Final questions}
We end with some open questions: 
\begin{enumerate}
\item When $S=flip$, the compatibility condition for (an involutive) $\beta$ is non-trivial,
but nevertheless there are many solutions (see Remark \ref{remflipinvo}).
Is there a characterization in ``involutive" terms? e.g. in terms of the dot operation 
(cyclic set structure), or brace,  associated to involutive solutions as considered by Rump
 \cite{R}?
\item Is it possible to classify connected virtual pairs in group theoretical terms?
\item Given a finite virtual pair $(X,S,\beta)$, it is easy to produce an algorithm 
computing generators and relations of  $\Uncfg(X)$ , but one needs to do case by case.
Is there a way to compute $\Uncfg(X)$ in general at least for a family of virtual
pairs? e.g. for $S$=biAlexander switch and $\beta$ affine?
\item When $\beta$=flip and $g\equiv 1$, then the conditions on $f$ are the same as
 the 2-cocycle condition considered in \cite{FG2}, which is a generalization of the
 quandle case considered in \cite{CEGS}. Also, in \cite{CEGS}, the authors prove that
the noncommutative 2-cocycle invariant (in the quandle case, for classical knots/links)
 is a quantum invariant. It seems that the fact that
this noncommutative invariant is a quantum one
may be generalized
to the biquandle case (and still classical knots or links), but it is not clear 
at all how to proceed when there 
are virtual crossings. It would be interesting to see what should be
the ``quantum algebraic'' categorical data corresponding to virtual pairs and 
2-cocycle pairs.
\end{enumerate}

{\bf Acknowledgements:} We would like to thank Andrew Bartholomew for 
kindly providing us the list of  virtual links with 4, 5 and 6 crossings and 
for attentive e-mail conversations.


\begin{thebibliography}{99}


\bibitem[BF]{BF}
A. Bartholomew and R. Fenn.
{\em Biquandles of Small Size and some Invariants of Virtual and Welded Knots},
Journal of Knot Theory Its Ramifications 20, No. 7, 943-954 (2011).
 See also
 http://www.layer8.co.uk/maths/biquandles/index.htm 
 

 \bibitem[CN]{CN}
 J. Ceniceros and S. Nelson. {\em Virtual Yang-Baxter cocycle invariants}, Transactions of the American Mathematical Society. Volume 361, No. 10 (2009),  5263-5283.
 
\bibitem[CEGS]{CEGS}
J. S. Carter, M. El Hamdadi, M. Gra\~na and M. Saito,
{\em Cocycle knot invariants from quandle 
modules and generalized quandle homology}.
Osaka J. Math. 42, No. 3, 499-541 (2005).
  


\bibitem[FG1]{FG} M. Farinati and J. Garc\'ia Galofre, 
 http://mate.dm.uba.ar/\~{}mfarinat/papers/GAP/virtual

\bibitem[FG2]{FG2} M. Farinati and J. Garc\'ia Galofre. {\em Link and knot invariants from non-abelian Yang-Baxter 2-cocycles},
Journal of Knot Theory and Its Ramifications  Vol 25 Nro. 13 (2016).


\bibitem[GAP2015]{GAP2015}
The GAP Group, GAP -- Groups, Algorithms, and Programming, Version 4.7.8; 2015. http://www.gap-system.org

\bibitem[GPV]{GPV}
 M. Goussarov, M. Polyak and O. Viro. 
{\em Finite-type invariants of classical and virtual knots},
Topology 39 (2000) 1045-1068.




\bibitem[K]{K} L. Kauffman.{\em Virtual Knot Theory}, Europ. J. Combinatorics (1999) 20, 663–691.

\bibitem[K2]{Kself} L. Kauffman. {\em A self-linking invariant of virtual knots},
Fundam. Math. 184, 135-158 (2004).



\bibitem[R]{R} W. Rump. {\em Braces, radical rings, and the quantum Yang-Baxter equation},
J. Algebra 307 (2007), 153–170.


\end{thebibliography}
\end{document}